%% file: paper_AP.tex
\documentclass[12pt]{article}

\usepackage{graphicx}
\usepackage{latexsym,amsmath,amsfonts,amscd, amsthm}
\usepackage{changebar}
\usepackage{color}
\usepackage{bm}
\usepackage{tikz}
\usepackage{multirow}
\usetikzlibrary{arrows,backgrounds,snakes}

\newtheoremstyle{plainNoItalics}{}{}{\normalfont}{}{\bfseries}{.}{ }{}

\theoremstyle{plain}
\newtheorem{thm}{Theorem}[section]

\theoremstyle{plainNoItalics}

\newtheorem{defn}[thm]{Definition}
\newtheorem{rem}[thm]{Remark}
\newtheorem{prop}[thm]{Proposition}

\newcommand{\beq}{\begin{equation}}
\newcommand{\eeq}{\end{equation}}
\newcommand{\beqa}{\begin{eqnarray}}
\newcommand{\eeqa}{\end{eqnarray}}
\newcommand{\bit}{\begin{itemize}}
\newcommand{\eit}{\end{itemize}}
\newcommand{\bedef}{\begin{defn}}
\newcommand{\edefn}{\end{defn}}
\newcommand{\bpro}{\begin{prop}}
\newcommand{\epro}{\end{prop}}

\newcommand{\df}{\partial}

\newcommand{\Dt}{\Delta t}

\begin{document}



\input{title}
\input{intro}
\input{macro_micro}
\input{DG_IMEX}
\input{simulations}

\input{conclusion}
\input{acknowledge}

\bibliographystyle{siam}
\bibliography{refer}

\end{document}

%% file: title.tex
\begin{center}
{\bf
High Order Asymptotic Preserving DG-IMEX Schemes
for Discrete-Velocity
Kinetic Equations in a Diffusive Scaling
}
\end{center}
\vspace{.2in}

\centerline{
Juhi Jang \footnote{Department of Mathematics, University of California Riverside, Riverside, CA 92521. E-mail: juhijang@math.ucr.edu. Supported in part by NSF grants DMS-0908007 and DMS-1212142.}
Fengyan Li \footnote{Department of Mathematical Sciences, Rensselaer Polytechnic Institute, Troy, NY 12180. E-mail: lif@rpi.edu. Supported in part by NSF CAREER award DMS-0847241.}
Jing-Mei Qiu \footnote{Department of Mathematics, University of Houston,
Houston, 77204. E-mail: jingqiu@math.uh.edu.
The third and fourth authors are supported in part by Air Force Office of Scientific Computing YIP grant FA9550-12-0318, NSF grant DMS-0914852 and DMS-1217008.}
Tao Xiong \footnote{Department of
Mathematics, University of Houston, Houston, 77204. E-mail:
txiong@math.uh.edu}
}

\bigskip
\centerline{\bf Abstract}

\smallskip
In this paper, we develop a family of high order asymptotic preserving schemes for some discrete-velocity kinetic equations
under a diffusive scaling, that in the asymptotic limit lead to macroscopic models
such as the heat equation, the porous media equation, the advection-diffusion equation, and the viscous Burgers' equation. Our approach is based on the 
micro-macro reformulation of the kinetic equation which involves a natural decomposition of the equation to the equilibrium and non-equilibrium parts.
To achieve high order accuracy and uniform stability as well as to capture the correct asymptotic limit, two new ingredients are employed in the proposed methods: discontinuous Galerkin spatial discretization of arbitrary order of accuracy with suitable numerical fluxes;
 high order globally stiffly accurate implicit-explicit Runge-Kutta scheme in time equipped with a properly chosen implicit-explicit strategy. Formal asymptotic analysis shows that the proposed scheme in the limit of $\varepsilon \rightarrow 0$ is an explicit, consistent and high order
discretization for the limiting equation.
Numerical results are presented to demonstrate the stability and high order accuracy of the proposed schemes together with their performance in the limit.

\vfill

\newpage

%% file: intro.tex
\section{Introduction}
\label{sec1}
\setcounter{equation}{0}
\setcounter{figure}{0}
\setcounter{table}{0}

The evolution of particles in rarefied gas dynamics, neutron transport, radiative transfer or stellar dynamics can be modeled at different levels of scale. The microscopic particle dynamics is described by Newton's laws of motion, while the macroscopic hydrodynamics can model the
observables such as density, velocity or temperature. Kinetic theory concerns the statistical description of particles via the particle density distribution rather than individual particles. The number of particles is typically more than $10^{20}$ and the computation, for instance in nanotechnology, is very expensive. In such situations, kinetic theory can be used to capture important properties of microscopic phenomena with reasonable computational cost.  On the other hand, kinetic models often provide more detailed description than the 
macroscopic ones.

Kinetic theory is at the center of multi-scale modeling connecting the invisible microscopic models  with the macroscopic models. In particular, when the mean free path of particles is sufficiently small, the system is close to the equilibrium state and it can be shown that a macroscopic model is a good approximation to the kinetic equation. Building a passage from kinetic to macroscopic models is a very interesting problem and there has been a lot of mathematical progress over the decades \cite{bardos1991fluid,  saint2009hydrodynamic}. 

Due to the relatively small yet multiple scales of particle interactions, developing efficient and effective kinetic simulation tools is mathematically and numerically challenging. For instance, non-thermal chemical equilibrium flow often displays multiple scales, requiring a hierarchy of physical models that range from kinetic theory to continuum fluid dynamics. To address multi-scale phenomena, there are hybrid domain decomposition method \cite{bourgat1994coupling, degond2005smooth}, asymptotic preserving (AP) method \cite{jin2010asymptotic, pareschi2011efficient}, moment method \cite{levermore1996moment, coulombel2005diffusion, carrillo2008numerical} etc. A goal shared by many of these multi-scale approaches is to resolve the physics of interest, e.g. macroscopic quantities, while avoiding the computational cost as much as possible to resolve small scale structures. In particular, for kinetic equation with a broad range of Knudsen number $\varepsilon$  that characterizes the kinetic scale, AP methods are designed to be uniformly stable with respect to $\varepsilon$, while mimicking the asymptotic limit from the kinetic to the hydrodynamic models on the PDE level as $\varepsilon$ goes to 0. As a result, the scheme in the limit of $\varepsilon\rightarrow0$ becomes a consistent discretization of the limiting macro-scale equations. The methods have been actively pursued by many researchers in recent years since they can effectively deal with multi-scales and capture hydrodynamic macro-scale limits in a uniform setting. We refer readers to \cite{jin2010asymptotic, pareschi2011efficient} for a review of the subject.

The focus of this paper is to design high order AP schemes for some discrete-velocity kinetic equations \cite{jin1998diffusive} under a diffusive scaling. It was shown in \cite{jin1996numerical, naldi1998numerical} that improper treatment of spatial discretization even with a stable implicit time discretization may fail to capture the correct asymptotic limit, when the spatial and temporal mesh sizes do not resolve the $\varepsilon$-scale. In \cite{jin1996numerical},  by building in the correct asymptotic behavior, a scheme capturing the correct asymptotic limit with under-resolved mesh size was designed. In \cite{naldi1998numerical}, a proper splitting between the convection and stiff source terms was introduced; later the scheme was coupled with a second order Runge-Kutta splitting in \cite{jin1998diffusive, jin2000uniformly}. In \cite{klar1998asymptotic}, an AP scheme was designed based on a standard perturbation procedure, followed by a fractional step scheme with a semi-implicit procedure. Based on the micro-macro decomposition, a first order finite difference   AP method was formulated in \cite{lemou2010new} for staggered grids, and its stability and error estimates were established in \cite{liu2010analysis}. Some new splitting strategies were proposed in \cite{boscarino2011implicit, boscarino2013flux}, with the focus on the design of various high order globally stiffly accurate implicit-explicit (IMEX) schemes.

In this paper, we propose a family of high order schemes for the discrete-velocity kinetic equations as an initial effort in designing high order simulation tools for more general kinetic equations under different scalings. These schemes are defined for a reformulated equation set which is obtained from a micro-macro decomposition of the problem. Based on  this reformulation,  discontinuous Galerkin (DG) spatial discretization of arbitrary order of accuracy is applied with suitable numerical fluxes, and in time, we employ globally stiffly accurate high order IMEX Runge-Kutta (RK) methods \cite{boscarino2011implicit} equipped with a carefully chosen implicit-explicit strategy.
Formal asymptotic analysis shows that the proposed methods, as Knudsen number $\varepsilon$ goes to $0$,
become explicit and consistent  high order discretizations for the limiting macro-scale equations. These limiting schemes specifically involve (local) DG spatial discretizations and explicit RK methods in time. Numerical results further demonstrate the stability and high order accuracy of the proposed schemes when $\varepsilon$ is of order $1$ and in the limit of $\varepsilon$ goes to $0$.  Both linear and nonlinear problems are considered with some smooth and non-smooth solutions. In a companion paper \cite{JLQX_analysis}, some theoretical results will be established in terms of uniform stability of the schemes, as well as the high order accuracy for smooth solutions.

Micro-macro decomposition is originated by the theoretical PDE community (for instance, see \cite{liu2004boltzmann}) in trying to solve the collisional kinetic equations such as Boltzmann equation and it has been successfully used to extract the structure of dissipation and the interaction between the equilibrium and the non-equilibrium parts. We advocate the micro-macro decomposition framework as it provides a general guidance on how to  perform spatial discretization with proper inter-element treatments, and on how to discretize in time with a suitable implicit-explicit strategy for different terms to achieve desired uniform stability.  It is based on projection, and therefore provides a natural framework for projection-based discretizations such as DG methods. DG methods are a class of finite element methods which use discontinuous approximating functions, they have been designed for a wide range of equations and gain popularity in many areas of science and engineering \cite{cockburn1998local,hesthaven2008nodal,riviere2008DG}. DG methods provide a framework to systematically design schemes with arbitrary order of accuracy. Some other attractive properties include their flexibility with general meshes and local approximations hence suitability for h-p adaptivity, compactness, being highly parallelable, ease of handling various boundary conditions, and provable stability and error estimates for many linear and nonlinear problems  \cite{JiangShu1994,cockburn2001runge}.
For high-order differential equations such as the diffusion equation, local DG methods were developed based on rewriting the system into its first order form \cite{cockburn1998local, xu2010local}. For the stationary radiative transfer equations, an AP scheme based on an upwind DG discretization was analyzed in \cite{guermond2010asymptotic}. To discretize in time, we apply  globally stiffly accurate IMEX RK methods \cite{boscarino2011implicit} which rely on an implicit-explicit strategy different from that in \cite{lemou2010new}. IMEX schemes were developed, analyzed, and applied to hyperbolic system with relaxation and in the diffusive limit in \cite{pareschi2005implicit, boscarino2008error, boscarino2010class, boscarino2011implicit}. It was shown that high order temporal accuracy can be achieved for both $\varepsilon=O(1)$ and $\varepsilon<<1$. Our IMEX strategy and the globally stiffly accurate property of IMEX schemes ensure the correct asymptotic limit of numerical solutions from internal stages of RK methods and at each discrete time.

The paper is organized as follows. In Section \ref{sec2}, we introduce the discrete-velocity kinetic equation,  and provide examples we will consider as well as their  micro-macro decomposition and diffusive limits. In Section \ref{sec:DG-IMEX:3}, a family of (formally) high order schemes are proposed, they employ DG spatial discretization with suitable numerical fluxes and globally stiffly accurate IMEX RK temporal discretization. Formal asymptotic analysis is then performed when $\varepsilon\rightarrow 0$ for the proposed methods.
In Section \ref{sec:numerical}, numerical results are presented. High order accuracy is observed as designed with $\varepsilon$'s ranging from $10^{-6}$ to $O(1)$, together with the correct asymptotic behavior for $\varepsilon\rightarrow 0$. Finally, conclusions are made in Section \ref{sec6}. 

%% file: macro_micro.tex
\section{Formulation}
\label{sec2}
\setcounter{equation}{0}
\setcounter{figure}{0}
\setcounter{table}{0}

We consider the following discrete-velocity kinetic model in a diffusive scaling
\begin{equation}\label{f-JH}
\varepsilon \partial_t f + v\cdot\nabla_x f = \frac1\varepsilon \mathcal{C}(f)
\end{equation}
with the initial data $f_0$ and suitable boundary conditions, where $f=f(x,v,t)$ is the distribution function of particles that depends on time $t>0$, position $x\in\Omega_x\subset {\mathbb R}$, and velocity $v\in \{1,-1\}$.
The parameter $\varepsilon>0$ measures the distance of the system to the equilibrium state and it can be regarded  as  the mean free path of the particles; when $\varepsilon$ is small, the system is close to equilibrium; when $\varepsilon$ is large, the system is far from equilibrium. $\mathcal{C}(f)$ is a collision operator that describes the interactions of particles among themselves and with the medium.

The following are interesting examples of $\mathcal{C}(f)$ that we will focus in this paper:
\begin{subequations}
\label{eq:Col}
\begin{align}
\label{Col_tele}  \mathcal{C}(f)&= \langle f\rangle -f ,\\
\label{Col_porous} \mathcal{C}(f)&=K\langle f \rangle^m \left( \langle f\rangle -f \right), \;\; K>0 \text{ and } m\leq 0, \\
\label{Col_adv_diff} \mathcal{C}(f)&= \langle f\rangle -f  +A\varepsilon v \langle f \rangle,\;\; |A\varepsilon|<1,  \\
\label{Col_bur}\mathcal{C}(f)&= \langle f\rangle -f  + C\varepsilon \left[ \langle f\rangle^2 -  ( \langle f\rangle -f )^2 \right] v,\;\; C>0.
\end{align}
\end{subequations}
Here $\langle f\rangle:=\int f d \mu$ where $d\mu$ is the discrete Lebesgue measure on $\{-1,1\}$:
\begin{equation*}\label{ave_f}
\langle f\rangle=\frac{f(x,v=1,t) + f(x,v=-1,t)}{2}.
\end{equation*}
We remark that the equation \eqref{f-JH} with \eqref{Col_tele}-\eqref{Col_bur} was extensively studied in \cite{jin1998diffusive}.
It was rewritten as the system of two equations for $u_1(x,t)=f(x,v=1,t)$ and $u_2(x,t)=f(x,v=-1,t)$. In particular, this new system with the collision operator \eqref{Col_tele} is known as the one-dimensional Goldstein-Taylor model or the telegraph equation.
In this paper, we will
design high order numerical schemes for \eqref{f-JH} with \eqref{Col_tele}-\eqref{Col_bur} based on its micro-macro decomposition.


\subsection{Micro-macro formulation}
\label{sec:2.2}

Note that $\langle 1\rangle =1$ and $\langle  v \rangle =0$ and we view 1 as the equilibrium state.
We consider the Hilbert space $L^2(d\mu)$ in $v$ variable with the following inner product:
$\langle f , g  \rangle := \int f g d\mu = \langle fg \rangle$ and introduce the orthogonal projection operator $\Pi$ onto $\text{Span} (1)$.
The micro-macro decomposition is performed in two steps:  first to write $f$ by using the orthogonal projections $\Pi$ and $\textbf{I}-\Pi$:
$f= \Pi f +(\textbf{I}-\Pi)f$
where $\Pi f$ is the macroscopic part (equilibrium part of $f$) and $(\textbf{I}-\Pi)f $ is the microscopic part (the non-equilibrium part of $f$) so that $\Pi f = \langle f\rangle \cdot 1$  and $\langle (\textbf{I}-\Pi)f, 1 \rangle=0$,
and then to decompose the kinetic equation via the projections in terms of $\Pi f$ and $(\textbf{I}-\Pi)f $.

In order to describe the micro-macro formulation for  \eqref{f-JH},
we introduce $ \rho:=  \langle f \rangle=\Pi f$ to denote the macroscopic density for $f$ and consider the following orthogonal decomposition,
\beq
\label{eq:f:mmd:LJ}
f = \langle f \rangle  + \varepsilon g = \rho+\varepsilon g
\eeq
where $\langle g\rangle = 0$.  Note that $\langle  \mathcal{C}(f)\rangle=0$ for \eqref{Col_tele}--\eqref{Col_bur}. Then the projection $\Pi$ of \eqref{f-JH}
gives rise to
$\varepsilon\partial_t \langle f\rangle + \partial_x\langle vf \rangle =0$. On the other hand,
$\langle g\rangle=0$ and $\langle v\rho\rangle= \rho\langle v\rangle=0$, hence we obtain
\[
\partial_t\rho+ \partial_x\langle vg \rangle  =0.
\]
Next, we apply $(\textbf{I}-\Pi)$
to \eqref{f-JH} and get
\[
\varepsilon^2 \partial_tg +v\partial_x\rho +\varepsilon (\mathbf{I}-\Pi)(v\partial_x g) = \frac{1}{\varepsilon}
(\mathbf{I}-\Pi) \mathcal{C}(f)= \frac{1}{\varepsilon}
 \mathcal{C}(\rho+\varepsilon g).
\]
Here  $\langle  \mathcal{C}(f)\rangle=0$ is used.  The micro-macro decomposition now yields the set of two equations
\begin{equation}
\begin{split}\label{eq:mmd:LJ}
&\partial_t\rho+\partial_x \langle vg \rangle  =0,\\
&\partial_t g  +\frac{1}{\varepsilon}  (\mathbf{I}-\Pi)(v\partial_xg) +\frac{1}{\varepsilon^2 }v\partial_x\rho = \frac{1}{\varepsilon^3 } \mathcal{C}(\rho+\varepsilon g).
\end{split}
\end{equation}
The solvability  for $f$ of $\eqref{f-JH}$ and the solvability for $\rho$ and $g$ of \eqref{eq:mmd:LJ} are equivalent.

For our two-velocity model, by writing
$j(x,t):=\frac{1}{2\varepsilon}(f(x,v=1,t)-f(x,v=-1,t))$,
the micro-macro decomposition \eqref{eq:mmd:LJ} can be put into the following form:
\beq
\label{eq:classical:LJ}
\begin{array}{l}
\partial_t\rho  + \partial_x j  = 0\,,\;\;\varepsilon^2 \partial_t j + \partial_x \rho  =  \mathcal{S}(\rho, j) 
\end{array}
\eeq
where $  \mathcal{S}(\rho, j)$ depends on $\mathcal{C}(f)$.

\subsection{{Diffusive limit and some energy identities }}

In this subsection, we will derive the equations that $\rho$ would satisfy in the limit of $\varepsilon\rightarrow 0$. {Energy identities will also be given for some cases.}

\bigskip
\noindent $\bullet$ {\underline{Heat and porous media equations.} }
We will treat \eqref{Col_tele} and \eqref{Col_porous}
together by regarding \eqref{Col_tele} as a special case of \eqref{Col_porous} with $m=0$ and $K=1$.
It is clear that
\beq\label{C:porus1}
 \mathcal{C}(f)=\mathcal{C}(\rho+\varepsilon g)=-K\varepsilon\rho^m g.
\eeq

To see the limiting equation for $\rho$,
we take the second equation in \eqref{eq:mmd:LJ}, write it as
$
g= - \frac{1}{K(1-m)}v\partial_x(\rho^{1-m}) +O(\varepsilon)$ 
and plug it into the first equation in \eqref{eq:mmd:LJ} to obtain
\[
 \partial_t\rho =
  \frac{1}{K(1-m)} \partial_{xx}\left(\rho^{1-m}\right)  + O(\varepsilon). 
\]
This will lead to a family of diffusion equations for $\rho$ as $\varepsilon\rightarrow 0$. If $m=0$, the limiting equation is the linear heat equation; if $m<0$,  it is the slow (nonlinear) diffusion equation, known as the porous media equation. 

With $j$ notation, we see that $\mathcal{S}(\rho,j)=-K\rho^mj$ and it is straightforward to check that the solutions to \eqref{eq:classical:LJ} in a periodic domain satisfy the following energy identity:
\[
\frac{1}{2}\frac{d}{dt} \int (\rho^2+\varepsilon^2 j^2) dx + \int K\rho^mj^2 dx=0
\]
which indicates the dissipation mechanism of the system.


\bigskip
\noindent
$\bullet$ {\underline{Advection-diffusion equation.} }From \eqref{Col_adv_diff},
we see that
\beq\label{C:Ad1}
 \mathcal{C}(f)=\mathcal{C}(\rho+\varepsilon g)=- \varepsilon (g-A v \rho).
 \eeq
By writing $g= Av\rho - v\partial_x\rho+O(\varepsilon )$ and plugging it into the first equation in \eqref{eq:mmd:LJ}, we recover an advection-diffusion equation for $\rho$ as $\varepsilon\rightarrow 0$:
\begin{equation}
\label{eq:ad:limit}
 \partial_t\rho +A\partial_x\rho=\partial_{xx}\rho.
\end{equation}

As in the previous case, some dissipation mechanism is expected, but we haven't found any specific results written in the literature. In the following, we present the $L^2$ stability result.

\begin{prop}
\label{prop:Ad_diff} Suppose $\rho$ and $j$ satisfy \eqref{eq:classical:LJ} with $\mathcal{S}(\rho,j)=-j+A\rho$ in a periodic domain. Then $\rho$ and $j$ obey the following energy identity:
\beq\label{Ad:energy}
\frac{1}{2}\frac{d}{dt}\int \left[ (1-\varepsilon^2A^2) \rho^2+\varepsilon^2(j-A\rho)^2 \right] dx +\int (j-A\rho)^2 dx=0.
\eeq
\end{prop}
\begin{proof}
Since $j-A\rho$ rather than $j$ drives the system dissipate, we will try to derive the energy identity for $\rho$ and $j-A\rho$. Multiply the first equation in \eqref{eq:classical:LJ} by $\rho$ and the second equation by $j-A\rho$, and integrate over $x$ to get
\[
\begin{split}
\underbrace{\int  \partial_t\rho  \rho dx}_{=\frac{1}{2}\frac{d}{dt}\int \rho^2 dx } + \underbrace{\int  \varepsilon^2 \partial_t j (j-A\rho)dx }_{(a)}+ \underbrace{\int\partial_x j \rho dx
+\int \partial_x\rho  (j-A\rho)dx}_{(b)} = -\int (j-A\rho)^2 dx.
\end{split}
\]
Notice that
\[
\begin{split}
(a)&=\int \varepsilon^2 \partial_t (j-A\rho) (j-A\rho)dx+\int  \varepsilon^2 A\partial_t\rho  j dx-\int  \varepsilon^2 A\partial_t\rho  A\rho dx\\
&=\frac{1}{2}\frac{d}{dt}\int  \varepsilon^2\left[(j-A\rho)^2-(A\rho)^2\right ]dx -\underbrace{\int \varepsilon^2 A\partial_x j  j dx}_{=0}\,,\\
(b)&=\int \partial_x( j\rho-\frac{A\rho^2 }{2}) dx =0 \;\;\text{(integration by parts)}.
\end{split}
\]
Hence, we obtain \eqref{Ad:energy}.  \end{proof}

 We observe that as long as $|\varepsilon A|<1$, the positivity of the energy in \eqref{Ad:energy} is guaranteed. The energy identity \eqref{Ad:energy} gives rise to the $L^2$ energy law for the advection-diffusion equation $\partial_t\rho+A\partial_x\rho=\partial_{xx}\rho$ in the limit of $\varepsilon\rightarrow 0$.

\bigskip
\noindent
$\bullet$ {\underline{Viscous Burgers' equation.} } For \eqref{Col_bur},
we notice that
\beq\label{C:B1}
 \mathcal{C}(f)=\mathcal{C}(\rho+\varepsilon g)=- \varepsilon g+C\varepsilon \left[ \rho^2 -  \varepsilon^2 g ^2 \right] v.
\eeq
We write $g= Cv\rho^2 - v\partial_x\rho+O(\varepsilon )$ and plug it into the first equation in \eqref{eq:mmd:LJ} to get
\[
\partial_t \rho +C\partial_x(\rho^2)=\partial_{xx}\rho  + O(\varepsilon),
\]
which yields a viscous Burgers' equation for $\rho$ as $\varepsilon\rightarrow 0$.

The corresponding equations \eqref{eq:classical:LJ} can be obtained with $\mathcal{S}(\rho,j)= -j+ C[\rho^2-\varepsilon^2 j^2 ]$. This model is known as the nonlinear Ruijgrok-Wu model \cite{ruijgrok1982completely}. 

%% file: DG_IMEX.tex
\section{DG-IMEX Methods}
\label{sec:DG-IMEX:3}

\setcounter{equation}{0}
\setcounter{figure}{0}
\setcounter{table}{0}

\newcommand{\eps}{\varepsilon}
\newcommand{\mD}{{\mathcal D}}
\newcommand{\Ox}{{\Omega_x}}
\newcommand{\Ov}{{\Omega_v}}
\newcommand{\xL}{{x_{i-\frac{1}{2}}}}
\newcommand{\xR}{{x_{i+\frac{1}{2}}}}
\newcommand{\iL}{{i-\frac{1}{2}}}
\newcommand{\iR}{{i+\frac{1}{2}}}
\newcommand{\testR}{{\phi}}   
\newcommand{\testG}{{\psi}}   
\newcommand{\mC}{{\mathcal{C}}}
\newcommand{\bI}{{\bf{I}}}

In this section, we will propose a family of high order
methods for the discrete-velocity kinetic equation \eqref{f-JH} based on its micro-macro 
reformulation \eqref{eq:mmd:LJ}.
The methods involve discontinuous Galerkin (DG) discretization of arbitrary order of accuracy in space and globally stiffly accurate implicit-explicit (IMEX) Runge-Kutta (RK) methods in time. 
Formal asymptotic analysis is performed to show that the proposed schemes in the limit of $\varepsilon\rightarrow0$ become explicit, consistent and high order schemes for the limiting equations.

\subsection{DG spatial discretizations}
\label{sec:3.1}

We first discretize the micro-macro system \eqref{eq:mmd:LJ} in space by DG methods. Since the discrete-velocity models considered here do not admit boundary layers, for brevity of presentation, we assume the boundary condition in $x$ is periodic and $\Omega_x=[x_\textrm{min},x_\textrm{max}]$. Other types  of boundary conditions can be easily treated.
Our spatial discretization is formulated on one mesh.
 Let's first introduce some notations. Start with $\{\xR\}_{i=0}^{i=N}$, a partition of $\Omega_x$. Here $x_{\frac12}=x_\textrm{min}$, $x_{N+\frac12}=x_\textrm{max}$, each element is denoted as $I_i=[\xL, \xR]$ with its length $\Delta x_i$, and $\Delta x=\max_i\Delta x_i$. Given any non-negative integer $k$, we define a finite dimensional discrete space
\begin{equation}
U_h^k=\left\{u\in L^2(\Omega_x): u|_{I_i}\in P^k(I_i), \forall i\right\}.
\label{eq:DiscreteSpace:1mesh}
\end{equation}
The local space $P^k(I)$ consists of polynomials of degree at most $k$ on $I$. Note functions in $U_h^k$ are piecewise-defined, and they are double-valued at grid points. For such functions, notations are introduced for jump and average: with $u(x^\pm)=\lim_{\Delta x\rightarrow 0^\pm} u(x+\Delta x)$, the jump and the average of $u$ at $x_\iR$ are defined as $[u]_\iR={u(x_\iR^+)-u(x_\iR^-)}$ and $\{u\}_\iR=\frac12(u(x_\iR^+)+u(x_\iR^-))$, respectively. We also use $u_\iR=u(x_\iR)$, $u^\pm_\iR=u(x^\pm_\iR), \forall i$.

We are now ready to define the semi-discrete DG method for the micro-macro system \eqref{eq:mmd:LJ}. Look for $\rho_h(\cdot,t), g_h(\cdot, v, t) \in U_h^k$, such that $\forall \testR, \testG \in U_h^k$, and $\forall i$,
\begin{subequations}
\label{eq:SDG}
\begin{align}
\int_{I_i} \df_t \rho_h \testR d x
&-  \int_{I_i} \langle vg_h\rangle \df_x\testR dx+ \widehat{\langle vg_h\rangle}_{\iR} {\testR^-_\iR} - \widehat{\langle vg_h\rangle}_{\iL} {\testR^+_\iL}=0, \label{eq:SDG:1}\\
\int_{I_i} \df_t g_h \testG d x
&+ \frac{1}{\eps}\int_{I_i} (\bI-\Pi) \mD_h(g_h; v)\; \testG dx
- \frac{1}{\eps^2} \left(\int_{I_i} v \rho_h \df_x\testG dx - v \widehat{\rho}_{h,\iR}{\testG^-_\iR} + v\widehat{\rho}_{h,\iL}{\testG^+_\iL} \right) \notag\\
&=\frac{1}{\eps^3}\int_{I_i} \mC (\rho_h+\eps g_h) \testG dx.   \label{eq:SDG:2}
\end{align}
\end{subequations}

In \eqref{eq:SDG:2}, $\mD_h(g_h; v)\in U_h^k$, and it is determined by an upwind discretization of $v \df_x g$ within the DG framework,
\begin{align}
(\mD_h(g_h; v), \testG)&=\sum_i\left(-\int_{I_i} vg_h \df_x\testG dx+\widetilde{(vg_h)}_\iR \testG_\iR^- - \widetilde{(vg_h)}_\iL\testG_\iL^+\right)\notag\\
&=-\sum_i\left(\int_{I_i} vg_h \df_x\testG dx\right) - \sum_i \widetilde{(vg_h)}_\iL[\testG]_\iL, \quad \psi\in U_h^k,
\label{eq:mD}
\end{align}
where $\widetilde{vg}$ is an upwind numerical flux consistent to $vg$,
\beq
\label{eq:vg:upwind:L-1}
\widetilde{vg}:=
\left\{
\begin{array}{ll}
v g^-,&\mbox{if}\; v>0\\
v g^+,&\mbox{if}\; v<0
\end{array}
\right.
=v\{g\}-\frac{|v|}{2}[g].
\eeq
Here and below, the standard inner product $(\cdot, \cdot)$ for the $L^2(\Omega_x)$ space is used, see the first term in  \eqref{eq:mD}.

Both $\widehat{\langle vg\rangle}$ and $\hat{\rho}$ in \eqref{eq:SDG} are also numerical fluxes, and they are consistent to the physical ones $\langle vg\rangle$ and $\rho$.  In this paper, we consider the following three choices:
%
\begin{subequations}
\label{eq:flux}
\begin{align}
\label{eq:flux:1}
{\textrm{alternating left-right:}} & \;\;\;\;  \widehat{\langle vg\rangle} = {\langle vg\rangle}^-, \hat{\rho} = {\rho}^+\;, \\
\label{eq:flux:11}
{\textrm{alternating right-left:}} & \;\;\;\; \widehat{\langle vg\rangle}={\langle vg\rangle}^+, \hat{\rho}={\rho}^-\;,\\
\label{eq:flux:2}
{\textrm{central:}} & \;\;\;\; \widehat{\langle vg\rangle}=\{\langle vg\rangle\}, \hat{\rho}=\{\rho\}\;.
\end{align}
\end{subequations}

Compared with central flux, alternating flux will result in a scheme with smaller dependant stencil. For different collision kernels $\mathcal{C}(f)$ in \eqref{eq:Col}, choices of fluxes may vary with the consideration of the stability and accuracy of the algorithm. More specifically,
\bit
\item[(1)] For the linear equation \eqref{eq:mmd:LJ} with \eqref{C:porus1} and $m=0$, one can use any numerical flux given in \eqref{eq:flux}.
\item[(2)]  For the nonlinear equation \eqref{eq:mmd:LJ} with \eqref{C:porus1} and $m\neq0$, as $\eps \rightarrow0$,  the limiting equation is the porous media equation for $\rho$. When $\rho>0$, the equation is diffusive. Similar to the linear case, we can choose any numerical flux given in \eqref{eq:flux}. However, the equation becomes degenerate when $\rho=0$, resulting in the phenomenon of finite speed propagation. Hence special attention has to be paid around the interface of $\rho=0$ when the alternating flux is used. In particular, when the interface is moving towards right, alternating left-right flux is used; when the interface is moving towards left, alternating right-left flux is used.
    More details are given in Section~\ref{sec:numerical}.
\item[(3)] For the convection-diffusion equations \eqref{eq:mmd:LJ} with \eqref{C:Ad1} or \eqref{C:B1},
     we choose the flux \eqref{eq:flux:1} for $A>0$ and for $C>0$ (resp. the flux \eqref{eq:flux:11} for $A<0$ and for $C<0$). Such choice ensures an upwind flux for the advective
     term in the limiting equation. When the convection is not dominating, central flux can also be used.
\eit

To get a more compact form of the scheme, one further sums up \eqref{eq:SDG} with respect to $i$,
\begin{subequations}
\label{eq:SDG:c}
\begin{align}
&(\df_t \rho_h, \testR)+a_h(g_h, \testR)=0, \label{eq:SDG:1:c}\\
&(\df_t g_h, \testG)+ \frac{1}{\eps} b_{h,v}(g_h, \testG)-\frac{v}{\eps^2} d_h(\rho_h, \testG)
=-\frac{1}{\eps^2}s^{(1)}_{h,v}(\rho_h, g_h, \testG)-s^{(2)}_{h,v}(g_h, \testG),   \label{eq:SDG:2:c}
\end{align}
\end{subequations}
where
\begin{subequations}
\begin{align}
\label{eq:ah}
a_h(g_h,\testR)&=-\sum_i \int_{I_i} \langle vg_h\rangle \df_x\testR dx - \sum_i \widehat{\langle vg_h\rangle}_{\iL} [\testR]_\iL,\\
\label{eq:bh}
b_{h,v}(g_h,\testG)&=((\bI-\Pi)\mD_h(g_h; v), \testG) =(\mD_h(g_h; v) - \langle\mD_h(g_h; v)\rangle, \testG),\\
\label{eq:dh}
d_h(\rho_h, \testG)&=\sum_i\int_{I_i} \rho_h \df_x\testG dx + \sum_i \widehat{\rho}_{h,\iL}[\testG]_\iL,
\end{align}
\end{subequations}
and
\begin{equation*}
s^{(1)}_{h,v}(\rho_h, g_h, \testG)=\left\{
\begin{array}{ll}
(K \rho_h^m g_h, \testG) &\;\;\;\; \textrm{for}\; \eqref{C:porus1}\\
(g_h-Av\rho_h, \testG) & \;\;\;\;\textrm{for}\; \eqref{C:Ad1}\\
(g_h-Cv\rho_h^2, \testG) &\;\;\;\; \textrm{for}\; \eqref{C:B1}
\end{array}
\right.,
s^{(2)}_{h,v}(g_h, \testG)=\left\{
\begin{array}{ll}
0 &\;\;\;\; \textrm{for}\; \eqref{C:porus1} \eqref{C:Ad1}\\
(Cvg_h^2, \testG)&\;\;\;\; \textrm{for}\; \eqref{C:B1}.
\end{array}
\right.
\end{equation*}
Note that the source term is written into a sum of two terms which are of different scales in $\eps$. These two terms will be treated differently in time discretization, see Section \ref{sec:3.2}.

\begin{rem} For the two-velocity models considered here, one can easily verify that $(\mathbf{I}-\Pi)(v\partial_xg)=v\langle \partial_x g\rangle$ holds. Recall the exact solution satisfies $\langle \partial_x g\rangle=\partial_x\langle g\rangle=0$,  hence $(\mathbf{I}-\Pi)(v\partial_xg)$ should vanish and it does not seem one needs to discretize $(\mathbf{I}-\Pi)(v\partial_xg)$ numerically. The analysis in the companion paper \cite{JLQX_analysis}, however, implies  that it is important to keep this term in order to obtain a more desirable stability condition on the time step, $\Dt=O(\eps h)$, in the convective regime with $\eps=O(1)$ instead of $\Dt=O(h^2)$ otherwise.
\end{rem}

\subsection{Fully discrete DG-IMEX methods}
\label{sec:3.2}

In this subsection, the semi-discrete DG method is further discretized in time with globally stiffly accurate IMEX RK schemes.
We start with  first order accuracy in time. Given $\rho_h^n(\cdot),\; g_h^n(\cdot, v)\in U_h^k$ that approximate the solution $\rho$ and $g$ at $t=t^n$, we look for $\rho_h^{n+1}(\cdot), \;g_h^{n+1}(\cdot,v) \in U_h^k$, such that $\forall \phi, \psi \in U_h^k$,
\begin{subequations}
\label{eq:FDG:1T}
\begin{align}
&\left(\frac{\rho_h^{n+1}-\rho_h^n}{\Dt}, \phi\right) + a_h(g_h^n, \phi)=0, \label{eq:FDG:1T:r}\\
&\left(\frac{g_h^{n+1}-g_h^n}{\Dt}, \psi\right)+ \frac{1}{\eps} b_{h,v}(g_h^n, \psi)-\frac{v}{\eps^2} d_h(\rho_h^{n+1}, \psi) =
 -\frac{1}{\eps^2}s^{(1)}_{h,v}(\rho_h^{n+1}, g_h^{n+1}, \psi)-s^{(2)}_{h,v}(g_h^{n}, \psi).
 \label{eq:FDG:1T:g}
\end{align}
\end{subequations}
We choose to use an implicit-explicit strategy different from that in \cite{lemou2010new}, as it is more natural to treat
both the collisional and convective stiff terms in the scale of $\frac{1}{\varepsilon^2}$  implicitly. This strategy will be used in higher order temporal discretizations discussed next.

To achieve higher order accuracy in time, we adopt the globally stiffly accurate IMEX RK schemes \cite{boscarino2011implicit}.
Recall an IMEX RK scheme can be represented with a double Butcher tableau
\beq
\label{eq: B_table}
\begin{array}{c|c}
\tilde{c} & \tilde{A}\\
\hline
 & {\tilde{b}}^T \end{array} \ \ \ \ \
\begin{array}{c|c}
{c} & {A}\\
\hline
 & {b^T} \end{array}~~,
\eeq
where $\tilde{A} = (\tilde{a}_{ij})$ is an $s\times s$  lower triangular matrix with zero diagonal entries for an explicit scheme, and $A = (a_{ij})$ is a lower triangular matrix with the same size but non-zero diagonal for a diagonally implicit RK (DIRK) method. The coefficients $\tilde{c}$ and $c$ are given by the usual relation $\tilde{c}_i = \sum_{j=1}^{i-1} \tilde a_{ij}$, $c_i = \sum_{j=1}^{i} a_{ij}$,
and vectors $\tilde{b} = (\tilde{b}_j)$ and $b = (b_j)$ provide the quadrature weights to combine internal stages of the RK method. The IMEX RK scheme is said to be {\em globally stiffly accurate} \cite{boscarino2011implicit} if
\begin{equation}
c_s = \tilde{c}_s = 1, \;\textrm{and}\;\;  A_{sj} = b_j, \tilde{A}_{sj} = \tilde{b}_j, \forall j=1, \cdots, s.
\label{eq:gsa}
\end{equation}
The first order IMEX scheme employed in \eqref{eq:FDG:1T}, represented by
\begin{displaymath}
\begin{array}{c|c c}
0 & 0&0\\
1 & 1 &0\\
\hline
 & 1 &0\\
 \end{array} \ \ \ \ \
\begin{array}{c|c c}
0 &0 & 0\\
1 &0&1\\
\hline
 &0&1\\
 \end{array}~~,
\end{displaymath}
is globally stiffly accurate.

Now we will apply a general globally stiffly accurate IMEX RK scheme, represented by \eqref{eq: B_table} with the property \eqref{eq:gsa}, to the semi-discrete DG scheme \eqref{eq:SDG:c}. This is combined with the same implicit-explicit strategy as in the first order case, namely, the terms $a_h$, $b_{h,v}$ and $s^{(2)}_{h, v}$ in \eqref{eq:SDG:c} are treated explicitly and the terms $d_h$, $s^{(1)}_{h, v}$ implicitly.
Given $\rho_h^n(\cdot),\; g_h^n(\cdot, v)\in U_h^k$, we look for $\rho_h^{n+1}(\cdot), \;g_h^{n+1}(\cdot,v) \in U_h^k$, such that $\forall \phi, \psi \in U_h^k$,
\begin{subequations}
\label{eq:FDG:GSA}
\begin{align}
\left(\rho_h^{n+1}, \phi\right)  &= \left(\rho_h^{n}, \phi\right) - \Delta t \sum_{l=1}^s \tilde{b}_l a_h(g_h^{(l)}, \phi), \label{eq:FDG:GSA:r}\\
\left(g_h^{n+1}, \psi\right) &= \left(g_h^{n}, \psi\right) - {\Delta t} \sum_{l=1}^{s}  \tilde{b}_l  \left(\frac{1}{\eps}b_{h,v}(g_h^{(l)}, \psi) + s^{(2)}_{h,v}(g_h^{(l)}, \psi)\right)\notag\\
& +\Delta t \sum_{l=1}^{s}  \frac{b_l}{\eps^2} \left (v d_h(\rho_h^{(l)}, \psi) - s^{(1)}_{h,v}(\rho_h^{(l)}, g_h^{(l)}, \psi) \right).
 \label{eq:FDG:GSA:g}
\end{align}
\end{subequations}
Here the approximations at the internal stages of an RK step, $\rho_h^{(l)}(\cdot), \;g_h^{(l)}(\cdot,v) \in U_h^k$ with  $l=1, \cdots, s$, satisfy
\begin{subequations}
\label{eq:FDG:GSA:stage}
\begin{align}
\left(\rho_h^{(l)}, \phi\right)  &= \left(\rho_h^{n}, \phi\right) - \Delta t \sum_{j=1}^{l-1} \tilde{a}_{lj} a_h(g_h^{(j)}, \phi), \label{eq:FDG:GSA:r:stage}\\
\left(g_h^{(l)}, \psi\right) &= \left(g_h^{n}, \psi\right) - {\Delta t} \sum_{j=1}^{l-1}  \tilde{a}_{lj}  \left(\frac{1}{\eps}b_{h,v}(g_h^{(j)}, \psi) + s^{(2)}_{h,v}(g_h^{(j)}, \psi)\right)\notag\\
& +\Delta t \sum_{j=1}^{l}  \frac{a_{lj}}{\eps^2} \left ( v d_h(\rho_h^{(j)}, \psi) - s_{h,v}^{(1)}(\rho_h^{(j)}, g_h^{(j)}, \psi) \right)
 \label{eq:FDG:GSA:g:stage}
\end{align}
\end{subequations}
for any $\phi, \psi \in U_h^k$.

The property of being globally stiffly accurate
guarantees that the updated numerical solution at $t^{n+1}$ is the same as the one from the last internal stage of one RK step. 
That is, one can equivalently replace equations in \eqref{eq:FDG:GSA} by
\begin{equation}
\label{eq:FDG:t}
\rho_h^{n+1} = \rho_h^{(s)}, \quad
g_h^{n+1} = g_h^{(s)},
\end{equation}
where $s$ is the total number of internal stages in the IMEX scheme. With the implicit treatment of stiff terms in \eqref{eq:FDG:GSA:stage}, in the limit of $\varepsilon\rightarrow 0$, equation \eqref{eq:FDG:GSA:g:stage} gives
\begin{equation}
\label{eq: qq1}
 v d_h(\rho_h^{(j)}, \psi) = s_{h,v}^{(1)}(\rho_h^{(j)}, g_h^{(j)}, \psi), \;\forall \psi\in U^k_h, \;\;\forall j=1,\cdots, s.
\end{equation}
The scheme being globally stiffly accurate further implies
\begin{equation}
\label{eq: qq2}
 v d_h(\rho_h^{n+1}, \psi) = s_{h,v}^{(1)}(\rho_h^{n+1}, g_h^{n+1}, \psi), \;\forall \psi\in U^k_h.
\end{equation}
Therefore the scheme projects the numerical solutions, both from internal stages and at discrete times, to the limiting  equilibrium when  $\varepsilon \rightarrow 0$ as in equation~\eqref{eq: qq1} and \eqref{eq: qq2} .


The second order globally stiffly accurate IMEX scheme used in this paper is the ARS(2, 2, 2) scheme \cite{ascher1997implicit} with a double Butcher Tableau
\begin{displaymath}
\begin{array}{c|c c c}
0 & 0&0&0\\
\gamma & \gamma&0&0\\
1 & \delta &1-\delta&0\\
\hline
& \delta &1-\delta&0\\
 \end{array} \ \ \ \ \
\begin{array}{c|c c c}
0 & 0&0&0\\
\gamma & 0&\gamma&0\\
1 &0 &1-\gamma&\gamma\\
\hline
&0 &1-\gamma&\gamma\\
 \end{array}
\end{displaymath}
where $\gamma = 1-\frac{1}{\sqrt{2}}$ and $\delta = 1-\frac{1}{2\gamma}$.
The third order one is
the ARS(4, 4, 3) scheme \cite{ascher1997implicit} with
\begin{displaymath}
\begin{array}{c|c c c c c}
0 & 0&0&0& 0 & 0\\
1/2 &1/2&0&0& 0 &0\\
2/3 &11/18&1/18&0& 0 &0\\
1/2 &5/6&-5/6&1/2& 0 &0\\
1 &1/4&7/4&3/4& -7/4 &0\\
\hline
&1/4&7/4&3/4& -7/4 &0\\
 \end{array} \ \ \ \ \
\begin{array}{c|c c c c c}
0 & 0&0&0& 0 & 0\\
1/2 &0&1/2&0&0 &0\\
2/3 &0&1/6&1/2& 0 &0\\
1/2 &0&-1/2&1/2& 1/2 &0\\
1 &0&3/2&-3/2& 1/2 &1/2\\
\hline
 &0&3/2&-3/2& 1/2 &1/2\\
 \end{array}
\end{displaymath}

\begin{rem}
The DG methods with the first order IMEX temporal discretization can be implemented by explicitly solving \eqref{eq:FDG:1T:r} for $\rho^{n+1}$, then solving for $g^{n+1}$ from a block-diagonal system defined by \eqref{eq:FDG:1T:g}. The higher order globally stiffly accurate IMEX scheme can be implemented in a stage-by-stage fashion similar to the first order case.
\end{rem}

\subsection{Formal asymptotic analysis}

In this subsection, we perform formal asymptotic analysis for the proposed DG-IMEX schemes as $\eps\rightarrow0$.

\begin{prop}
\label{prop: DG_IMEX:others}
Consider the DG-IMEX scheme \eqref{eq:FDG:GSA:stage}-\eqref{eq:FDG:t} for the micro-macro formulation \eqref{eq:mmd:LJ},
 with one choice of numerical fluxes in \eqref{eq:flux}, a globally stiffly accurate IMEX scheme \eqref{eq: B_table}, consistent initial condition, and periodic boundary condition. Then in the limit of $\varepsilon \rightarrow 0$,
\begin{itemize}
\item[({$\mathcal P$}1)]
the DG-IMEX scheme formally becomes: look for $\rho_h^{n+1}(\cdot), \;g_h^{n+1}(\cdot,v) \in U_h^k$, satisfying \eqref{eq:FDG:t}, while the solutions from the internal stages $\rho_h^{(l)}(\cdot), \;g_h^{(l)}(\cdot,v) \in U_h^k$ with  $l=1, \cdots, s$, satisfying
\begin{subequations}
\label{eq:FDG:GSA:stage:limit}
\begin{align}
\left(\rho_h^{(l)}, \phi\right)  &= \left(\rho_h^{n}, \phi\right) - \Delta t \sum_{j=1}^{l-1} \tilde{a}_{lj} a_h(g_h^{(j)}, \phi), \\
v d_h(\rho_h^{(l)}, \psi) &=
\left\{
\begin{array}{ll}
(K(\rho_h^{(l)})^m g_h^{(l)}, \testG) & \;\;\;\;\textrm{for}\; \eqref{C:porus1}\\
(g_h^{(l)}-Av\rho_h^{(l)}, \testG) & \;\;\;\;\textrm{for}\; \eqref{C:Ad1}\\
(g_h^{(l)}-Cv(\rho_h^{(l)})^2, \testG) &\;\;\;\; \textrm{for}\; \eqref{C:B1}
\end{array}\right.
\label{eq:FDG:GSA:g:stage:limit}
\end{align}
\end{subequations}
for any $\phi, \psi \in U_h^k$. This is a consistent scheme for the limiting equation with
the corresponding numerical flux in \eqref{eq:flux} and the explicit RK time discretization defined by $\tilde A$, $\tilde b$, $\tilde c$ in \eqref{eq: B_table}.

\item[({$\mathcal P$}2)]
If one further denotes $q=\langle v g\rangle$ (so to their approximations), then the limiting scheme, \eqref{eq:FDG:t} and \eqref{eq:FDG:GSA:stage:limit},  becomes an explicit and consistent local DG scheme for the heat and porous media equation, advection-diffusion equation, or viscous Burgers equation in its first order form as below,
\begin{equation}
\partial_t\rho+\partial_x q=0, \quad q=
\left\{
\begin{array}{ll}
-\frac{1}{K(1-m)}\partial_x (\rho^{1-m})~, & \;\;\;\;\textrm{for}\; \eqref{C:porus1}~,\\
A\rho-\partial_x \rho~, & \;\;\;\;\textrm{for}\; \eqref{C:Ad1}~,\\
C\rho^2-\partial_x \rho~, &\;\;\;\; \textrm{for}\;\eqref{C:B1}~.
\end{array}\right.\label{eq:rho:q}
\end{equation}

\end{itemize}
\end{prop}

\begin{proof}
 ({$\mathcal P$}1) can be obtained straightforwardly with the definition of $s_{h,v}^{(1)}(\rho_h, g_h, \psi)$ and by formally taking $\eps\rightarrow 0$ in \eqref{eq:FDG:GSA:stage}-\eqref{eq:FDG:t}. Note the time discretization comes directly from the explicit part of \eqref{eq: B_table}, and the numerical flux is carried over as well. To see that the limiting scheme gives a consistent discretization for the limiting equation, one only needs to recall from Section \ref{sec:2.2} that the limiting equation, when written in the form of $\rho$ and $g$, is given by
\begin{equation}
\partial_t\rho+\partial_x \langle v g\rangle=0,\quad g=
\left\{
\begin{array}{ll}
-\frac{1}{K(1-m)}v \partial_x (\rho^{1-m})~, & \;\;\;\;\textrm{for}\; \eqref{C:porus1}~,\\
Av\rho-v\partial_x \rho~, & \;\;\;\;\textrm{for}\; \eqref{C:Ad1}~,\\
Cv\rho^2-v\partial_x \rho~, &\;\;\;\; \textrm{for}\; \eqref{C:B1}~.
\end{array}\right.
\end{equation}

Next we multiply equation \eqref{eq:FDG:GSA:g:stage:limit} with $v$, apply the $\langle \cdot \rangle$ operator,
and denote $q=\langle v g\rangle$ (so to their approximations), then the limiting scheme becomes: look for
$\rho_h^{(l)}, q_h^{(l)} \in U_h^k$ with $l=1,\cdots, s$, such that
\begin{subequations}
\begin{align}
\left(\rho_h^{(l)}, \phi\right)  &= \left(\rho_h^{n}, \phi\right) - \Delta t \sum_{j=1}^{l-1} \tilde{a}_{lj} r_h(q_h^{(j)}, \phi), \quad \forall \phi \in U_h^k,\\
d_h(\rho_h^{(l)}, \psi) &=
\left\{
\begin{array}{ll}
(K(\rho_h^{(l)})^m q_h^{(l)}, \testG) & \;\;\;\;\textrm{for}\; \eqref{C:porus1}\\
(q_h^{(l)}-A\rho_h^{(l)}, \testG) & \;\;\;\;\textrm{for}\; \eqref{C:Ad1}\\
(q_h^{(l)}-C(\rho_h^{(l)})^2, \testG) &\;\;\;\; \textrm{for}\; \eqref{C:B1}
\end{array}\right.
\quad \forall \psi \in U_h^k,
\end{align}
\label{eq:LDG}
\end{subequations}
and $\rho_h^{n+1}, q_h^{n+1} \in U_h^k$  satisfying
\begin{equation}
\label{eq:LDG:3}
\rho_h^{n+1}=\rho_h^{(s)}, \quad q_h^{n+1}=q_h^{(s)}.
\end{equation}
Here $d_h$ is given in \eqref{eq:dh} and
\beq
\label{eq:rh}
r_h(q_h, \phi) = -\sum_i \int_{I_i} q_h \df_x\testR dx - \sum_i \hat{q}_{h,\iL} [\testR]_\iL,
\eeq
with $\hat{q}$ defined in the same fashion as $\langle v g\rangle$ (see \eqref{eq:flux}). This exactly gives a consistent local DG method for \eqref{eq:rho:q}.
 \end{proof}

\begin{rem}
(1) For the telegraph equation, the limit of the proposed scheme reproduces a local DG method studied in \cite{cockburn1998local} (with the time discretization as the explicit RK method defined by $\tilde A, \tilde b, \tilde c$)
for linear heat equation. It can be proved (see \cite{cockburn1998local} for part of the analysis) that the semi-discrete local DG scheme for the heat equation with any of the alternating fluxes achieves an optimal $(k+1)^{th}$ order of accuracy,
and the scheme with the central flux achieves a sub-optimal $k^{th}$ order for odd $k$ and an optimal $(k+1)^{th}$ order for even $k$, when piecewise polynomials of degree $k$ are used as approximations. This is consistent with our numerical observations in Section~\ref{sec:numerical}.

(2) If the numerical flux in \eqref{eq:flux} is chosen based on the sign of $A$ and $C$ as discussed in section \ref{sec:3.1}, then the convective term in the limiting equation will be discretized in an upwind fashion.

(3) In general, more rigorous analysis would be needed to show whether the limiting schemes, which are consistent with formal high order accuracy, are indeed stable and with good accuracy.
\end{rem}

%% file: simulations.tex
 \section{Numerical Examples}
\label{sec:numerical}
\setcounter{equation}{0}
\setcounter{figure}{0}
\setcounter{table}{0}

In this section, we will demonstrate the performance of the proposed schemes by applying them to several numerical examples. Two kinds of fluxes in (\ref{eq:flux}) will be used: alternating and central.
For alternating fluxes, without specifying, we would use (\ref{eq:flux:1}) which is referred to as  the ``left-right flux''. Likewise, \eqref{eq:flux:11} will be called the ``right-left flux''.  Our scheme is denoted as
``DG(k+1)-IMEX(k+1)'' if piecewise polynomials of degree at most $k$ are used in the discrete space \eqref{eq:DiscreteSpace:1mesh} together with the $(k+1)^{th}$ order IMEX scheme in time, for $k=0, 1, 2$.  The simulation is carried out up to the final time $T$ on uniform meshes with totally $N$ elements. The time step $\Delta t$ is determined by
\begin{equation*}
\Delta t = C_{hyper} \varepsilon \Delta x + C_{diff} \Delta x^2.
\end{equation*}
That is, $\Delta t=O(\varepsilon \Delta x)$ in the rarefied (convective) regime where $\varepsilon=O(1)$,
and $\Delta t=O(\Delta x^2)$ in the parabolic (diffusive) regime where $\varepsilon<<1$. We take $C_{hyper}=0.5$ and $C_{diff}=0.25$ for DG1-IMEX1, and this is based on a uniform stability analysis of the method for the telegraph equation \cite{JLQX_analysis}.  As for DG(k+1)-IMEX(k+1) with $k>0$,  $C_{hyper}$ and $C_{diff}$ are chosen based on the numerical simulations with $\varepsilon=O(1), 10^{-2}, 10^{-6}$. In particular, we take $C_{hyper}=0.5$ and $C_{diff}=0.01$ for DG2-IMEX2, and $C_{hyper}=0.25$ and $C_{diff}=0.006$ for DG3-IMEX3.


\subsection{Telegraph equation}
\label{sec:4.1}

Consider (\ref{C:porus1}) with $m=0$ and $K=1$.
We test the accuracy for our scheme with the following exact solution
\begin{eqnarray}
\begin{cases}
\rho(x,t)=\frac{1}{r}\exp(r t)\sin(x),\quad r=\frac{-2}{1+\sqrt{1-4\varepsilon^2}}, \\
j(x,t)=\exp(r t)\cos(x)
\end{cases}
\label{tele1}
\end{eqnarray}
on the domain $[-\pi, \pi]$ with periodic boundary conditions.
In Tables \ref{tab11}-\ref{tab13}, we show the errors and orders of accuracy with $\varepsilon=0.5, 10^{-2}, 10^{-6}$ for  DG(k+1)-IMEX(k+1), $k=0, 1, 2$, respectively. Alternating left-right flux is used and $T=1$.
 For all three $\varepsilon$s,
a uniform $(k+1)^{th}$ order of convergence is observed for DG(k+1)-IMEX(k+1).  We further present in Tables \ref{tab14}-\ref{tab16} the results from DG(k+1)-IMEX(k+1) using the central flux. For all three $\eps$s,
 a uniform $(k+1)^{th}$ order for even $k$ and $k^{th}$ order for odd $k$ is observed.


\begin{table}
\centering
\caption{$L^1$ errors and orders of $\rho$ and $j$ for the telegraph equation with the exact solution (\ref{tele1}), $T=1.0$, DG1-IMEX1 with left-right flux.}
\vspace{0.2cm}
  \begin{tabular}{|c|c|c|c|c|c|}
    \hline
    &N  &  $L^1$ error of $\rho$ & order   & $L^1$ error of $j$  & order \\\hline
\multirow{5}{*}{$\varepsilon=0.5$}&    10 &     6.04E-02 &       --&     7.46E-02 &       --  \\  \cline{2-6}
    &20 &     2.19E-02 &     1.46&     3.38E-02 &     1.14  \\  \cline{2-6}
    &40 &     9.20E-03 &     1.25&     1.60E-02 &     1.08  \\  \cline{2-6}
    &80 &     4.19E-03 &     1.14&     7.81E-03 &     1.03  \\  \cline{2-6}
   &160 &     2.00E-03 &     1.07&     3.86E-03 &     1.02  \\  \hline
\multirow{5}{*}{$\varepsilon=10^{-2}$}&    10 &     3.79E-02 &       --&     8.05E-02 &       --  \\  \cline{2-6}
   & 20 &     1.78E-02 &     1.09&     3.77E-02 &     1.09  \\  \cline{2-6}
   & 40 &     8.79E-03 &     1.02&     1.85E-02 &     1.03  \\  \cline{2-6}
   & 80 &     4.36E-03 &     1.01&     9.22E-03 &     1.01  \\  \cline{2-6}
   &160 &     2.17E-03 &     1.01&     4.60E-03 &     1.00  \\  \hline
\multirow{5}{*}{$\varepsilon=10^{-6}$}&    10 &     3.79E-02 &       --&     8.03E-02 &       --  \\  \cline{2-6}
   & 20 &     1.79E-02 &     1.08&     3.76E-02 &     1.09  \\  \cline{2-6}
   & 40 &     8.82E-03 &     1.02&     1.85E-02 &     1.02  \\  \cline{2-6}
   & 80 &     4.38E-03 &     1.01&     9.21E-03 &     1.01  \\  \cline{2-6}
   &160 &     2.18E-03 &     1.01&     4.60E-03 &     1.00  \\  \hline
  \end{tabular}
\label{tab11}
\end{table}

\begin{table}
\centering
\caption{$L^1$ errors and orders of $\rho$ and $j$ for the telegraph equation with the exact solution (\ref{tele1}), $T=1.0$, DG2-IMEX2 with left-right flux.}
\vspace{0.2cm}
  \begin{tabular}{|c|c|c|c|c|c|}
    \hline
    &N  &  $L^1$ error of $\rho$ & order   & $L^1$ error of $j$  & order \\\hline
\multirow{5}{*}{$\varepsilon=0.5$}&    10 &     1.35E-03 &       --&     2.36E-03 &       --  \\  \cline{2-6}
    &20 &     3.00E-04 &     2.17&     4.90E-04 &     2.27  \\  \cline{2-6}
    &40 &     7.23E-05 &     2.05&     1.14E-04 &     2.10  \\  \cline{2-6}
    &80 &     1.79E-05 &     2.01&     2.76E-05 &     2.04  \\  \cline{2-6}
   &160 &     4.46E-06 &     2.01&     6.82E-06 &     2.02  \\  \hline
\multirow{5}{*}{$\varepsilon=10^{-2}$}&    10 &     4.83E-03 &       --&     4.94E-03 &       --  \\  \cline{2-6}
    &20 &     1.19E-03 &     2.02&     1.19E-03 &     2.06  \\  \cline{2-6}
    &40 &     2.96E-04 &     2.01&     2.97E-04 &     2.00  \\  \cline{2-6}
    &80 &     7.40E-05 &     2.00&     7.40E-05 &     2.00  \\  \cline{2-6}
   &160 &     1.85E-05 &     2.00&     1.85E-05 &     2.00  \\  \hline
\multirow{5}{*}{$\varepsilon=10^{-6}$}&    10 &     4.82E-03 &       --&     4.93E-03 &       --  \\  \cline{2-6}
    &20 &     1.19E-03 &     2.02&     1.18E-03 &     2.06  \\  \cline{2-6}
    &40 &     2.96E-04 &     2.00&     2.96E-04 &     2.00  \\  \cline{2-6}
    &80 &     7.40E-05 &     2.00&     7.40E-05 &     2.00  \\  \cline{2-6}
   &160 &     1.85E-05 &     2.00&     1.85E-05 &     2.00  \\  \hline
  \end{tabular}
\label{tab12}
\end{table}

\begin{table}
\centering
\caption{$L^1$ errors and orders of $\rho$ and $j$ for the telegraph equation with the exact solution (\ref{tele1}), $T=1.0$, DG3-IMEX3 with left-right flux.}
\vspace{0.2cm}
  \begin{tabular}{|c|c|c|c|c|c|}
    \hline
    &N  &  $L^1$ error of $\rho$ & order   & $L^1$ error of $j$  & order \\\hline
\multirow{5}{*}{$\varepsilon=0.5$}&    10 &     6.33E-05 &       --&     9.48E-05 &       --  \\  \cline{2-6}
    &20 &     7.54E-06 &     3.07&     1.15E-05 &     3.04  \\  \cline{2-6}
    &40 &     9.31E-07 &     3.02&     1.44E-06 &     3.00  \\  \cline{2-6}
    &80 &     1.16E-07 &     3.01&     1.80E-07 &     3.00  \\  \cline{2-6}
   &160 &     1.44E-08 &     3.00&     2.24E-08 &     3.00  \\  \hline
\multirow{5}{*}{$\varepsilon=10^{-2}$}&    10 &     2.53E-04 &       --&     2.46E-04 &       --  \\  \cline{2-6}
    &20 &     3.11E-05 &     3.03&     3.11E-05 &     2.98  \\  \cline{2-6}
    &40 &     3.89E-06 &     3.00&     3.89E-06 &     3.00  \\  \cline{2-6}
    &80 &     4.87E-07 &     3.00&     4.87E-07 &     3.00  \\  \cline{2-6}
   &160 &     6.09E-08 &     3.00&     6.09E-08 &     3.00  \\  \hline
\multirow{5}{*}{$\varepsilon=10^{-6}$}&    10 &     2.53E-04 &       --&     2.46E-04 &       --  \\  \cline{2-6}
    &20 &     3.11E-05 &     3.03&     3.11E-05 &     2.98  \\  \cline{2-6}
    &40 &     3.89E-06 &     3.00&     3.89E-06 &     3.00  \\  \cline{2-6}
    &80 &     4.87E-07 &     3.00&     4.87E-07 &     3.00  \\  \cline{2-6}
   &160 &     6.09E-08 &     3.00&     6.09E-08 &     3.00  \\  \hline
  \end{tabular}
\label{tab13}
\end{table}


\begin{table}
\centering
\caption{$L^1$ errors and orders of $\rho$ and $j$ for the telegraph equation with the exact solution (\ref{tele1}), $T=1.0$, DG1-IMEX1 with central flux.}
\vspace{0.2cm}
  \begin{tabular}{|c|c|c|c|c|c|}
    \hline
    &N  &  $L^1$ error of $\rho$ & order   & $L^1$ error of $j$  & order \\\hline
\multirow{5}{*}{$\varepsilon=0.5$}&    10 &     2.49E-02 &       --&     3.80E-02 &       --  \\  \cline{2-6}
   & 20 &     9.80E-03 &     1.34&     1.74E-02 &     1.13  \\  \cline{2-6}
   & 40 &     4.42E-03 &     1.15&     8.17E-03 &     1.09  \\  \cline{2-6}
   & 80 &     2.07E-03 &     1.10&     3.99E-03 &     1.03  \\  \cline{2-6}
   &160 &     1.00E-03 &     1.04&     1.97E-03 &     1.02  \\  \hline
\multirow{5}{*}{$\varepsilon=10^{-2}$}&    10 &     4.14E-02 &       --&     3.46E-02 &       --  \\  \cline{2-6}
   & 20 &     1.89E-02 &     1.13&     1.74E-02 &     0.99  \\  \cline{2-6}
   & 40 &     9.04E-03 &     1.07&     8.71E-03 &     1.00  \\  \cline{2-6}
   & 80 &     4.43E-03 &     1.03&     4.36E-03 &     1.00  \\  \cline{2-6}
   &160 &     2.20E-03 &     1.01&     2.18E-03 &     1.00  \\  \hline
\multirow{5}{*}{$\varepsilon=10^{-6}$}&    10 &     4.12E-02 &       --&     3.46E-02 &       --  \\  \cline{2-6}
   & 20 &     1.89E-02 &     1.13&     1.74E-02 &     0.99  \\  \cline{2-6}
   & 40 &     9.01E-03 &     1.07&     8.69E-03 &     1.00  \\  \cline{2-6}
   & 80 &     4.42E-03 &     1.03&     4.34E-03 &     1.00  \\  \cline{2-6}
   &160 &     2.19E-03 &     1.01&     2.17E-03 &     1.00  \\  \hline
  \end{tabular}
\label{tab14}
\end{table}

\begin{table}
\centering
\caption{$L^1$ errors and orders of $\rho$ and $j$ for the telegraph equation with the exact solution (\ref{tele1}), $T=1.0$, DG2-IMEX2 with central flux.}
\vspace{0.2cm}
  \begin{tabular}{|c|c|c|c|c|c|}
    \hline
    &N  &  $L^1$ error of $\rho$ & order   & $L^1$ error of $j$  & order \\\hline
\multirow{5}{*}{$\varepsilon=0.5$}&    10 &     1.31E-02 &       --&     8.93E-03 &       --  \\  \cline{2-6}
   & 20 &     1.06E-02 &     0.30&     3.64E-03 &     1.29  \\  \cline{2-6}
   & 40 &     7.20E-03 &     0.56&     1.25E-03 &     1.54  \\  \cline{2-6}
   & 80 &     4.26E-03 &     0.76&     3.74E-04 &     1.74  \\  \cline{2-6}
   &160 &     2.33E-03 &     0.87&     1.03E-04 &     1.86  \\  \hline
\multirow{5}{*}{$\varepsilon=10^{-2}$}&    10 &     1.02E-02 &       --&     1.01E-02 &       --  \\  \cline{2-6}
   & 20 &     4.51E-03 &     1.18&     4.78E-03 &     1.09  \\  \cline{2-6}
   & 40 &     1.95E-03 &     1.21&     2.37E-03 &     1.01  \\  \cline{2-6}
   & 80 &     7.58E-04 &     1.37&     1.24E-03 &     0.94  \\  \cline{2-6}
   &160 &     1.53E-04 &     2.31&     6.90E-04 &     0.85  \\  \hline
\multirow{5}{*}{$\varepsilon=10^{-6}$}&    10 &     1.06E-02 &       --&     1.00E-02 &       --  \\  \cline{2-6}
   & 20 &     4.83E-03 &     1.13&     4.66E-03 &     1.10  \\  \cline{2-6}
   & 40 &     2.29E-03 &     1.08&     2.25E-03 &     1.05  \\  \cline{2-6}
   & 80 &     1.11E-03 &     1.04&     1.11E-03 &     1.03  \\  \cline{2-6}
   &160 &     5.50E-04 &     1.02&     5.48E-04 &     1.01  \\  \hline
  \end{tabular}
\label{tab15}
\end{table}

\begin{table}
\centering
\caption{$L^1$ errors and orders of $\rho$ and $j$ for the telegraph equation with the exact solution (\ref{tele1}), $T=1.0$, DG3-IMEX3 with central flux.}
\vspace{0.2cm}
  \begin{tabular}{|c|c|c|c|c|c|}
    \hline
    &N  &  $L^1$ error of $\rho$ & order   & $L^1$ error of $j$  & order \\\hline
\multirow{5}{*}{$\varepsilon=0.5$}&    10 &     6.56E-05 &       --&     7.13E-05 &       --  \\  \cline{2-6}
   & 20 &     7.10E-06 &     3.21&     7.76E-06 &     3.20  \\  \cline{2-6}
   & 40 &     8.14E-07 &     3.13&     9.81E-07 &     2.98  \\  \cline{2-6}
   & 80 &     9.49E-08 &     3.10&     1.11E-07 &     3.14  \\  \cline{2-6}
   &160 &     1.24E-08 &     2.93&     1.51E-08 &     2.88  \\  \hline
\multirow{5}{*}{$\varepsilon=10^{-2}$}&    10 &     1.77E-04 &       --&     1.76E-04 &       --  \\  \cline{2-6}
   & 20 &     2.05E-05 &     3.11&     2.04E-05 &     3.11  \\  \cline{2-6}
   & 40 &     2.49E-06 &     3.04&     2.48E-06 &     3.04  \\  \cline{2-6}
   & 80 &     3.07E-07 &     3.02&     3.05E-07 &     3.02  \\  \cline{2-6}
   &160 &     3.81E-08 &     3.01&     3.79E-08 &     3.01  \\  \hline
\multirow{5}{*}{$\varepsilon=10^{-6}$}&    10 &     1.76E-04 &       --&     1.76E-04 &       --  \\  \cline{2-6}
   & 20 &     2.04E-05 &     3.11&     2.04E-05 &     3.11  \\  \cline{2-6}
   & 40 &     2.48E-06 &     3.04&     2.48E-06 &     3.04  \\  \cline{2-6}
   & 80 &     3.05E-07 &     3.02&     3.05E-07 &     3.02  \\  \cline{2-6}
   &160 &     3.79E-08 &     3.01&     3.79E-08 &     3.01  \\  \hline
  \end{tabular}
\label{tab16}
\end{table}

Next we consider a Riemann problem for the telegraph equation, with initial conditions
\begin{eqnarray}
 \begin{cases}
 \rho_L=2.0, \quad j_L=0.0, \qquad -1<x<0, \\
 \rho_R=1.0, \quad j_R=0.0, \qquad  0<x<1,
 \end{cases}
 \label{tele2}
\end{eqnarray}
and inflow and outflow boundary conditions on the computational domain $[-1, 1]$. We depict in Figure \ref{fig1} the numerical solutions of $\rho$ and $j$ from DG(k+1)-IMEX(k+1), $k=0, 1, 2$ in the rarefied regime ($\varepsilon=0.7$) at $T=0.25$ and in the parabolic regime ($\varepsilon=10^{-6}$) at $T=0.04$. Alternating left-right flux and central flux are used. The mesh size is $\Delta x= 0.05$. The reference solutions are obtained using DG3-IMEX3 with the corresponding numerical flux on a more refined grid of $\Delta x=0.004$. One can see that the higher order scheme (DG3-IMEX3) has much better resolution than the lower order one (DG1-IMEX1). Compared with the results in \cite{jin1998diffusive}, the high order results can also catch the shock speed well for such a  coarse mesh. Small oscillations are observed for our scheme with any of the numerical fluxes in the rarefied regime, even for DG1-IMEX1, due to the dispersive error of the scheme. DG3-IMEX3 demonstrates better control over the numerical oscillations than DG2-IMEX2.

\begin{figure}[ht]
\centering
\includegraphics[totalheight=2.0in]{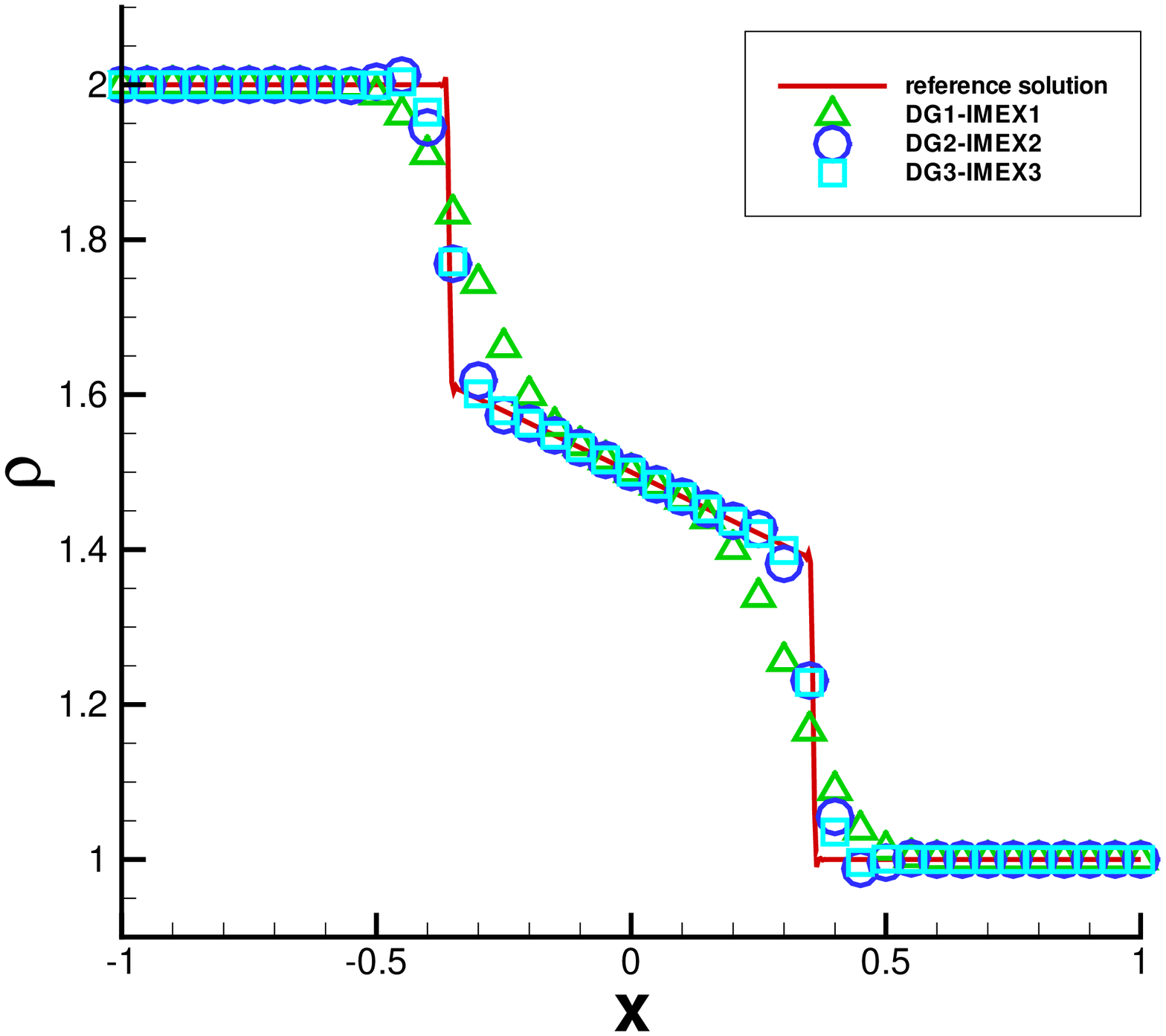},
\includegraphics[totalheight=2.0in]{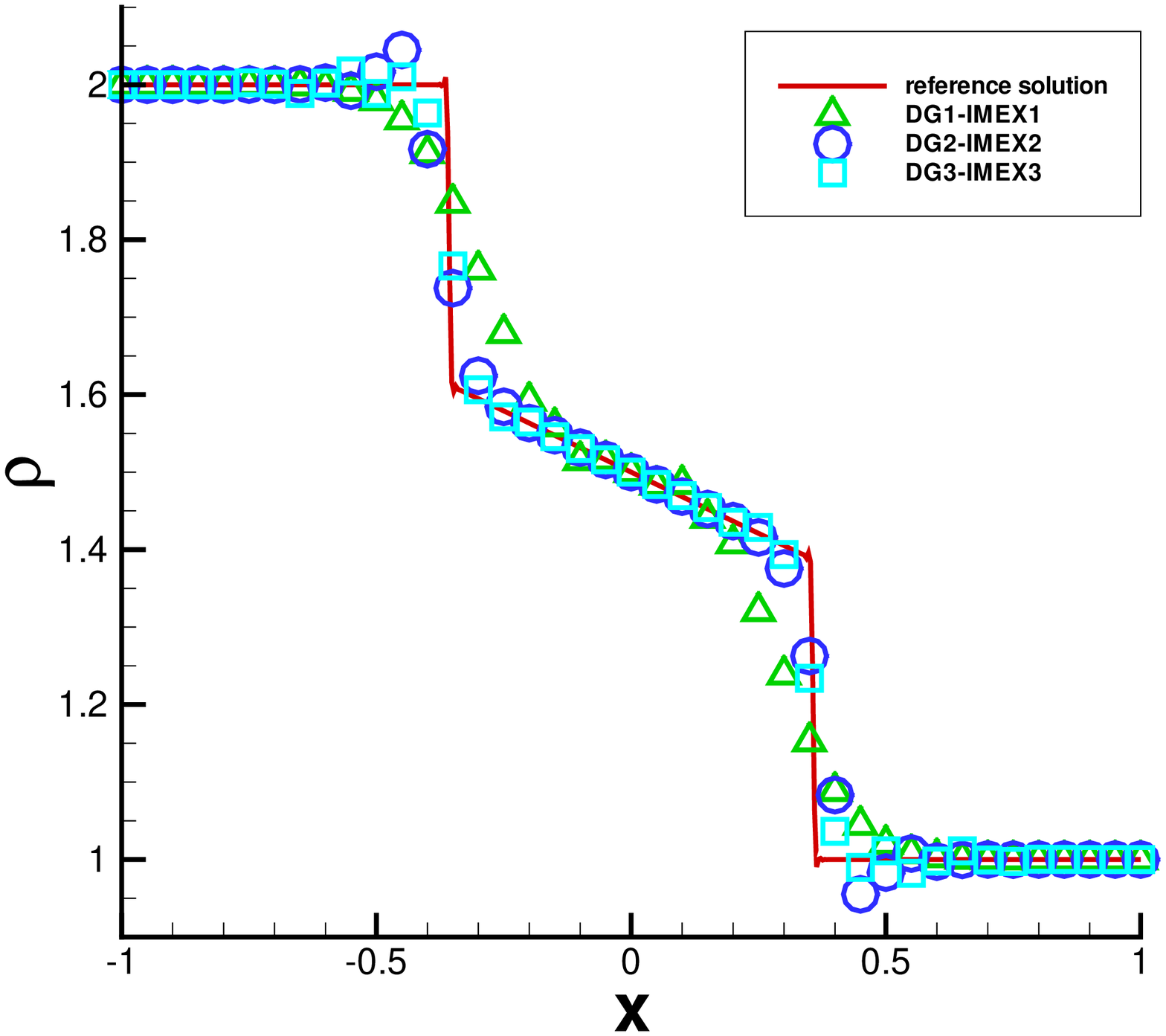}\\
\includegraphics[totalheight=2.0in]{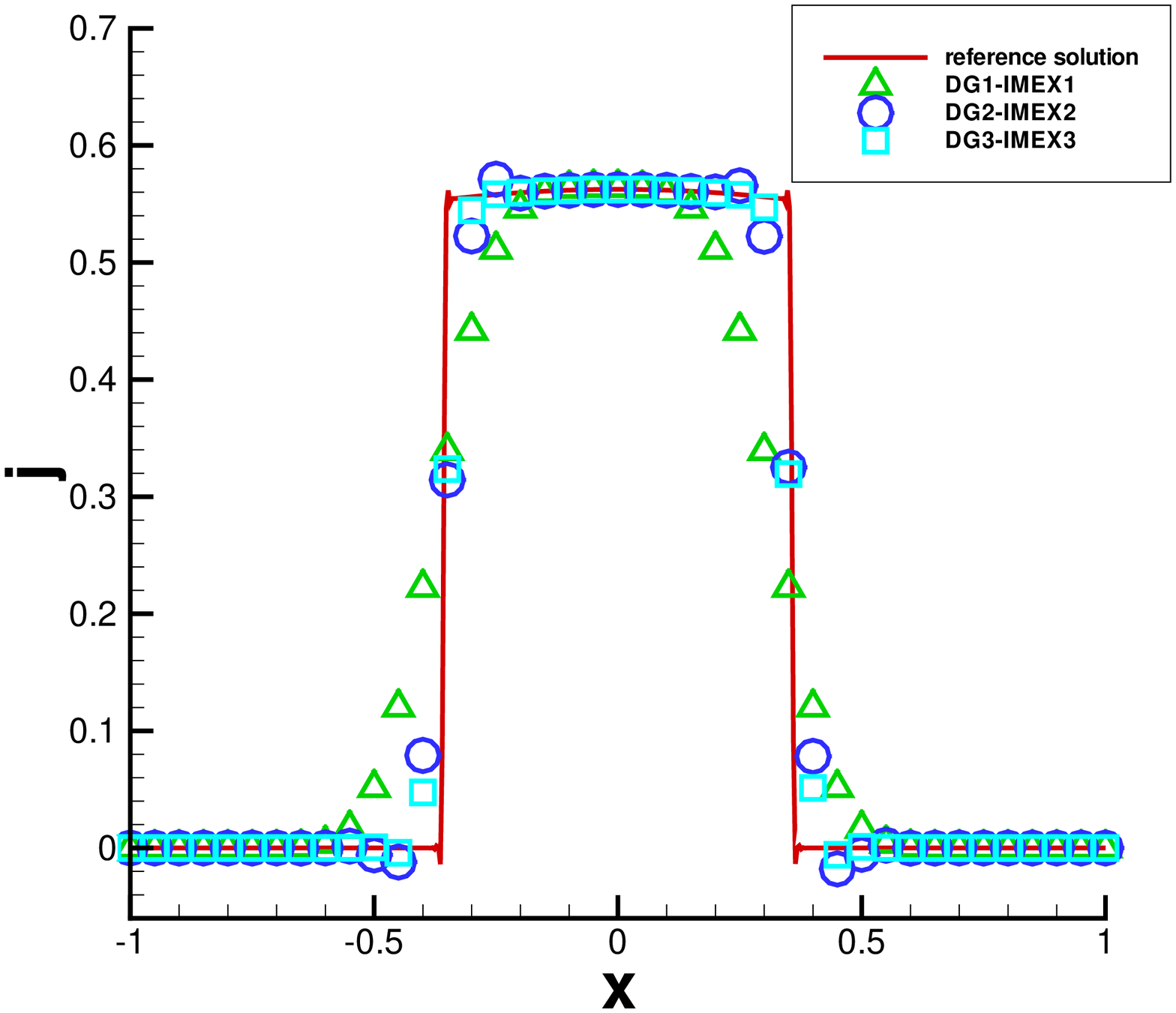},
\includegraphics[totalheight=2.0in]{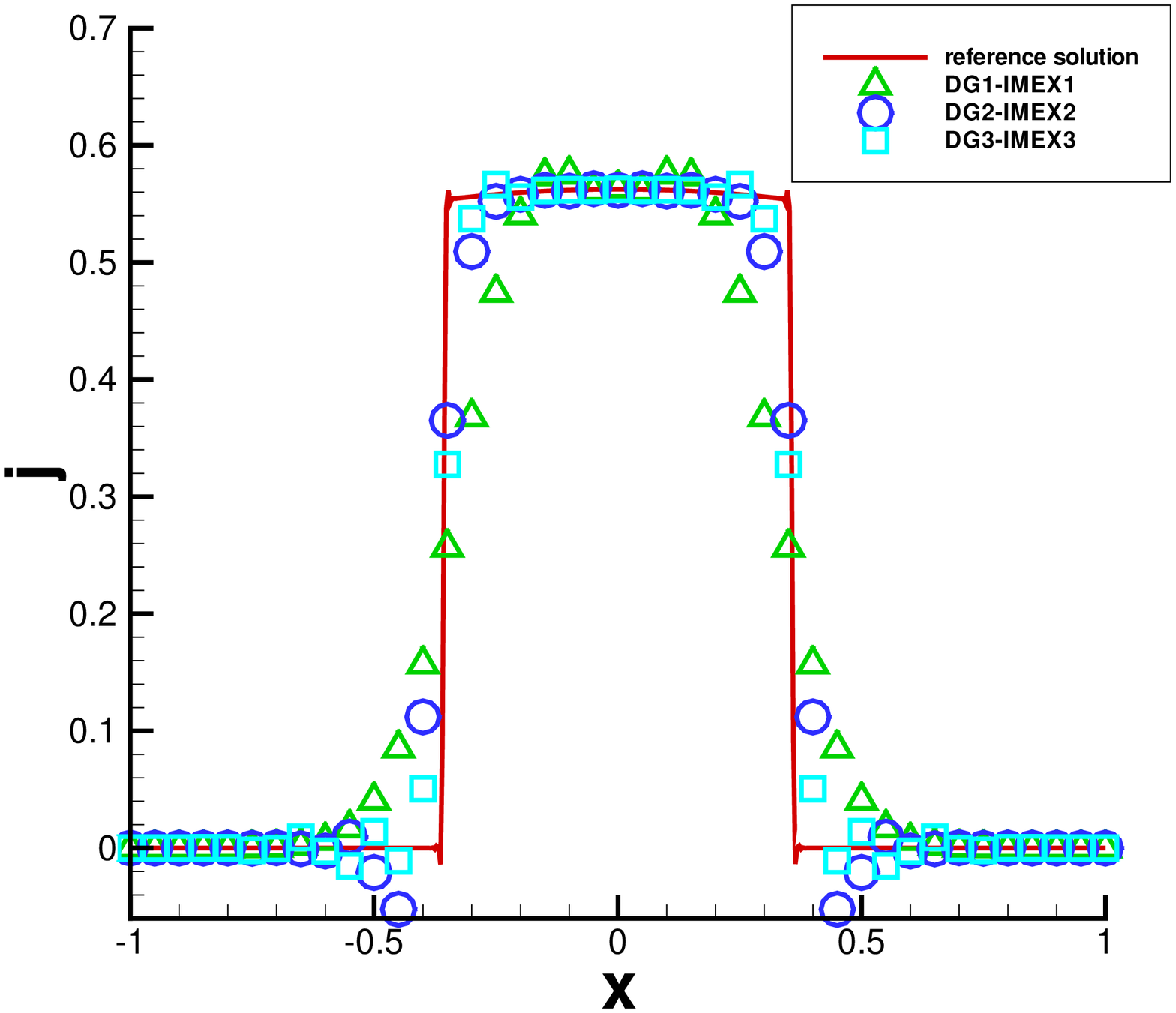}\\
\includegraphics[totalheight=2.0in]{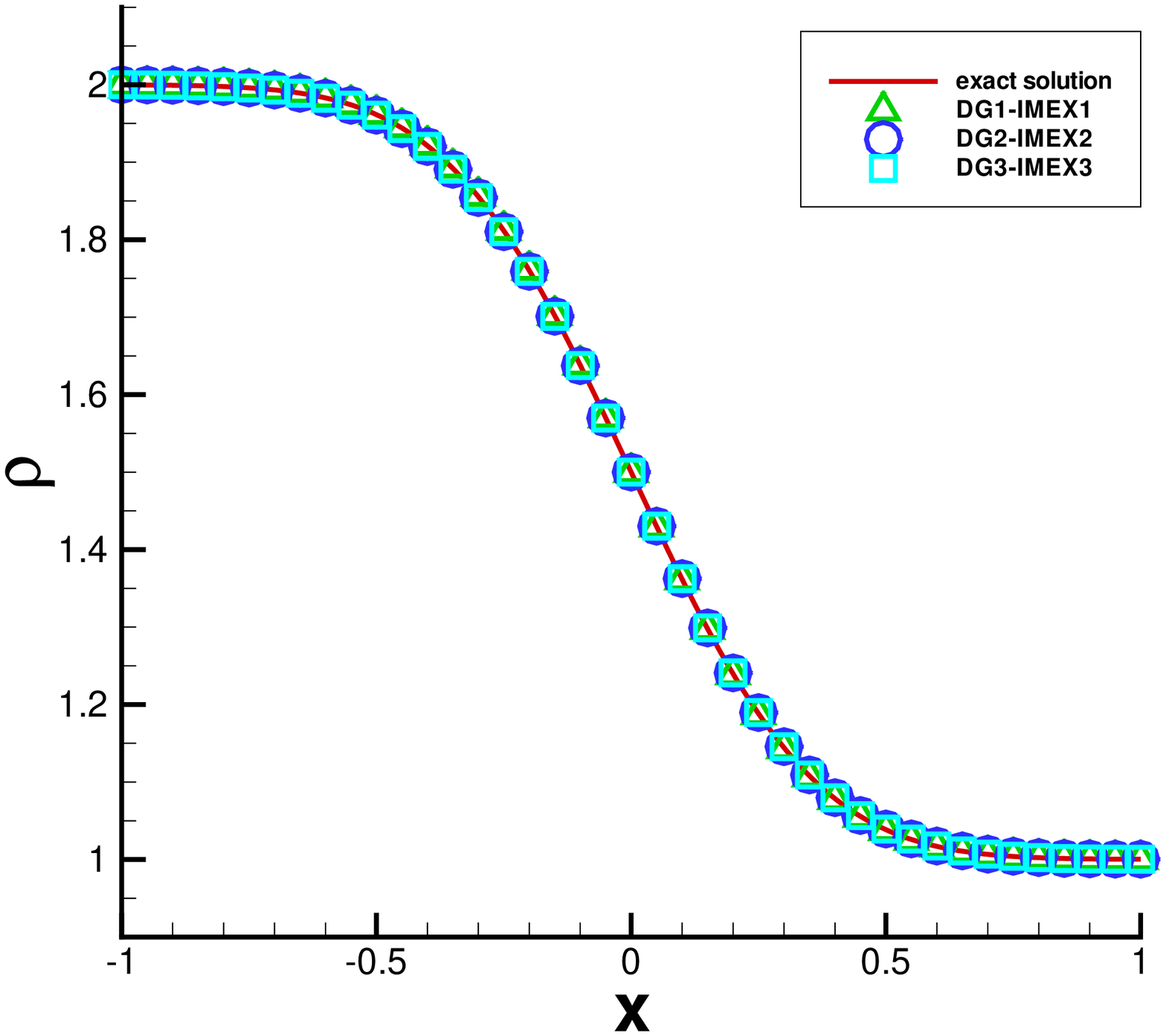},
\includegraphics[totalheight=2.0in]{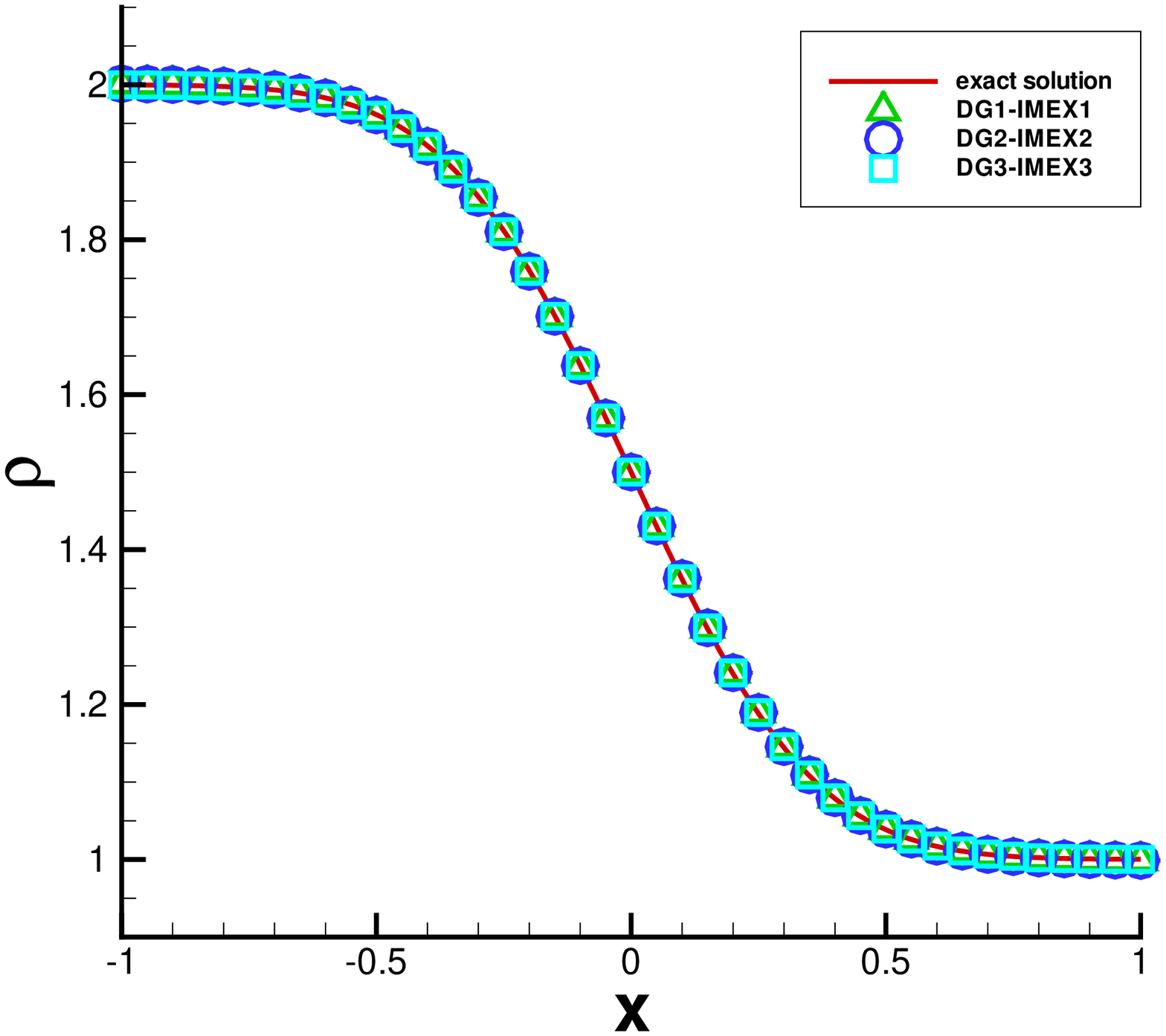}\\
\includegraphics[totalheight=2.0in]{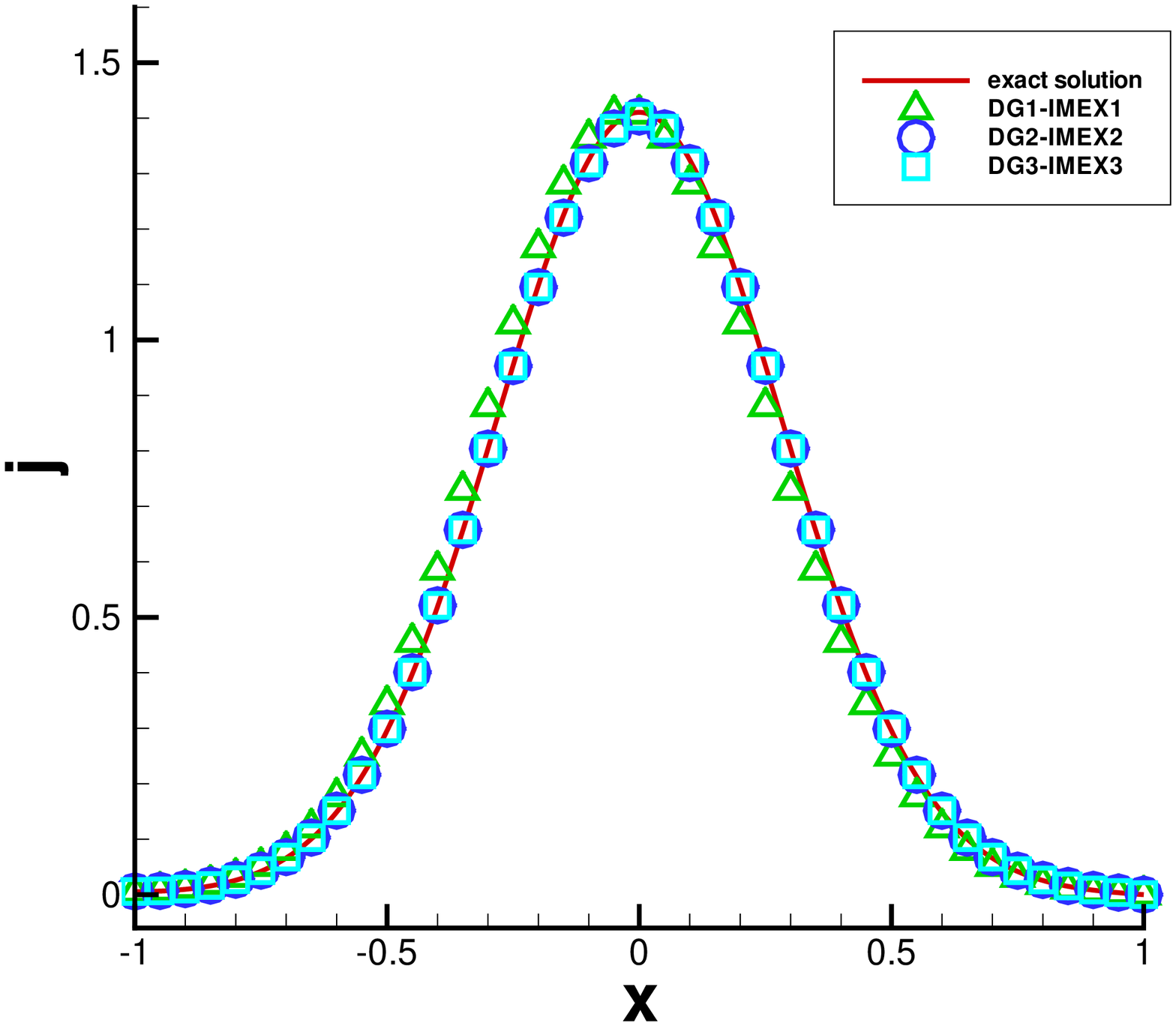},
\includegraphics[totalheight=2.0in]{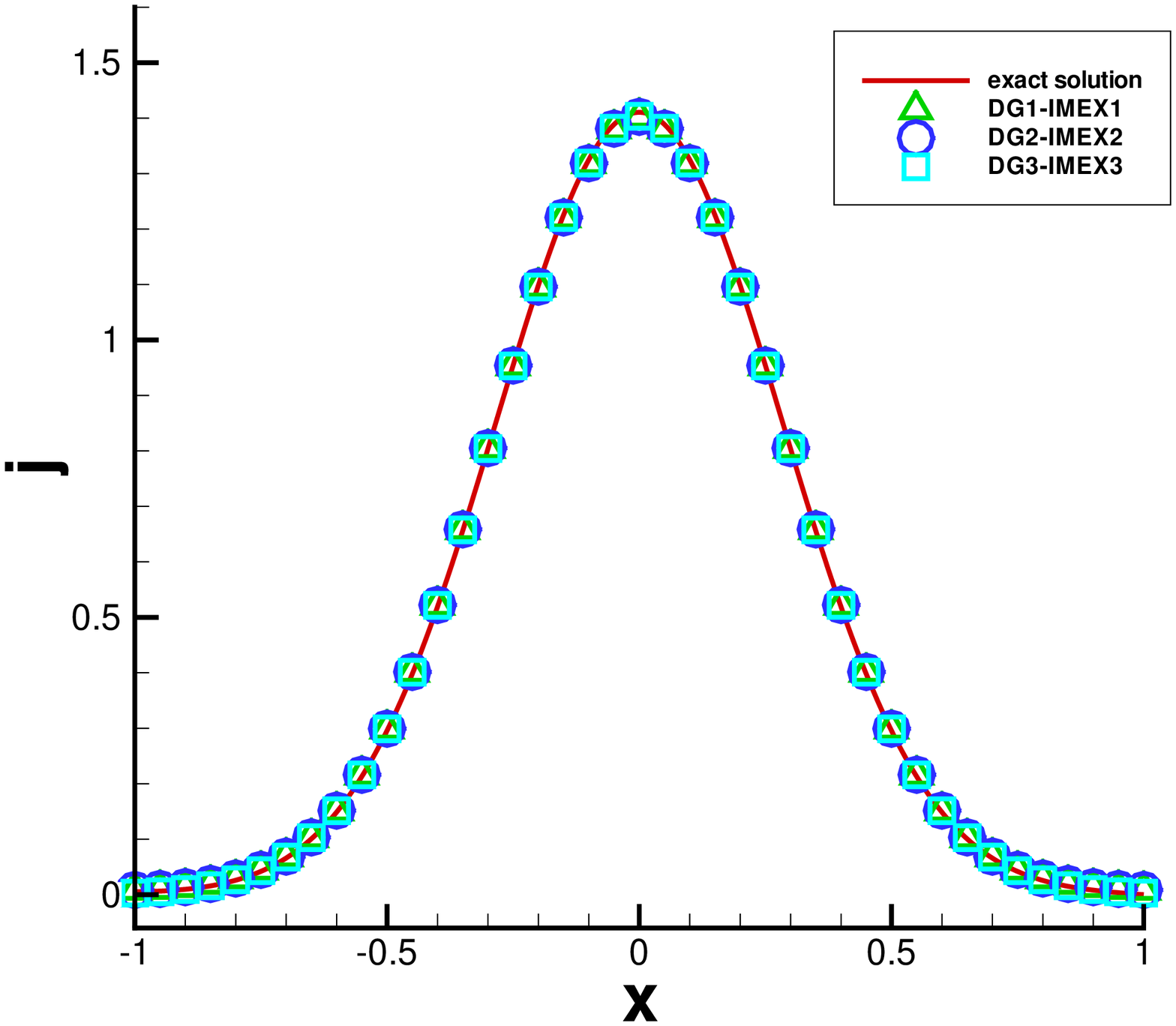}
\caption{Numerical solutions of the telegraph equation with initial conditions (\ref{tele2})
by DG(k+1)-IMEX(k+1), $k=0, 1, 2$ in the rarefied regime $\varepsilon=0.7$ at $T=0.25$ (top two rows) and in the parabolic regime $\varepsilon=10^{-6}$ at $T=0.04$ (bottom two rows). $\Delta x=0.05$.
Left: left-right flux; Right: central flux. The reference solutions are obtained by DG3-IMEX3 with $\Delta x=0.004$, with the left-right flux for the left column, and the central flux for the right column.}
\label{fig1}
\end{figure}

For this example, one can see that the scheme with alternating fluxes has better order of accuracy, see especially the case with odd $k$, than that using central flux when the solutions are smooth. For Riemann problem, the scheme with alternating flux also performs better with smaller numerical oscillation. On the other hand, the scheme using central flux can be more flexible for some other problems, see for example those in Section \ref{sec:4.4}.


\subsection{Advection-diffusion equation}
\label{sec:4.2}

Consider (\ref{C:Ad1}) with $A=1$. As $\varepsilon\rightarrow 0$,
this leads to the classical advection-diffusion equation \eqref{eq:ad:limit} as the limiting equation. By taking into account the convection term with $A=1$, we only use the left-right flux throughout this subsection.

Note that the limiting advection-diffusion equation  \eqref{eq:ad:limit} with $A=1$ admits the following exact solution
\begin{equation}
 \label{ad1}
 \rho(x,t)=e^{-t}\sin(x-t), \qquad j(x,t)=e^{-t}(\sin(x-t)-\cos(x-t)),
\end{equation}
on the domain $[-\pi, \pi]$ with periodic boundary conditions. We start with the initial conditions \eqref{ad1} at $t=0$, and implement our scheme with $\varepsilon=10^{-6}$ up to $T=0.1$. The numerical results are compared with \eqref{ad1},  with errors and convergence orders reported in Table \ref{tab2}. One can observe $(k+1)^{th}$ order of accuracy for DG(k+1)-IMEX(k+1), $k=0, 1, 2$.


\begin{table}
\centering
\caption{$L^1$ errors and orders of $\rho$ and $j$ for the advection-diffusion equation for $\varepsilon=10^{-6}$,  with the numerical solutions compared with the exact solutions \eqref{ad1} to the limiting equation, $T=0.1$, left-right flux. }
\vspace{0.2cm}
  \begin{tabular}{|c|c|c|c|c|c|}
    \hline
    &N  &  $L^1$ error of $\rho$ & order   & $L^1$ error of $j$  & order \\\hline
\multirow{5}{*}{DG1-IMEX1}&    10 &     9.41E-02 &       --&     2.03E-01 &       --  \\  \cline{2-6}
   & 20 &     4.62E-02 &     1.03&     9.94E-02 &     1.03  \\  \cline{2-6}
   & 40 &     2.30E-02 &     1.01&     4.98E-02 &     1.00  \\  \cline{2-6}
   & 80 &     1.15E-02 &     1.00&     2.50E-02 &     0.99  \\  \cline{2-6}
   &160 &     5.74E-03 &     1.00&     1.25E-02 &     1.00  \\  \hline
\multirow{5}{*}{DG2-IMEX2}&    10 &     1.03E-02 &       --&     1.67E-02 &       --  \\  \cline{2-6}
   & 20 &     2.71E-03 &     1.92&     4.10E-03 &     2.02  \\  \cline{2-6}
   & 40 &     7.01E-04 &     1.95&     1.03E-03 &     2.00  \\  \cline{2-6}
   & 80 &     1.79E-04 &     1.97&     2.57E-04 &     2.00  \\  \cline{2-6}
   &160 &     4.51E-05 &     1.99&     6.43E-05 &     2.00  \\  \hline
\multirow{5}{*}{DG3-IMEX3}&    10 &     6.05E-04 &       --&     8.57E-04 &       --  \\  \cline{2-6}
   & 20 &     7.62E-05 &     2.99&     1.08E-04 &     2.98  \\  \cline{2-6}
   & 40 &     9.56E-06 &     3.00&     1.36E-05 &     3.00  \\  \cline{2-6}
   & 80 &     1.20E-06 &     3.00&     1.69E-06 &     3.00  \\  \cline{2-6}
   &160 &     1.50E-07 &     3.00&     2.12E-07 &     3.00  \\  \hline
  \end{tabular}
\label{tab2}
\end{table}

Next we  consider a Riemann problem, with the initial conditions given as
\begin{eqnarray}
 \begin{cases}
   \rho_L=4.0, \quad j_L=0.0, \qquad -10<x<0, \\
   \rho_R=2.0, \quad j_R=0.0, \qquad   0<x<10,
 \end{cases}
 \label{ad2}
\end{eqnarray}
and inflow and outflow boundary conditions. The exact solution for the limiting advection-diffusion equation \eqref{eq:ad:limit} is
\begin{equation}
 \rho(x,t)=\frac{1}{2}(\rho_L+\rho_R)+\frac{1}{2}(\rho_L-\rho_R) erf(\frac{(t-x)}{2\sqrt{t}})
 \label{adex2}
\end{equation}
where $erf$ denotes the error function.
We compute the numerical solutions of DG(k+1)-IMEX(k+1) for $k=0, 1, 2$ with mesh size $\Delta x=0.5$ on the domain $[-10, 10]$ up to $T=3.0$. For $\varepsilon=0.5$, we compare the numerical solution $\rho$ and the reference solution obtained by DG3-IMEX3 on a much fine mesh ($\Delta x=0.04$) in Figure \ref{fig2} (Left). On the right of  Figure \ref{fig2}, we plot the numerical solution from the scheme with $\varepsilon=10^{-6}$, and the analytical solution \eqref{adex2}. In both regimes, numerical solutions match the reference or exact solutions very well. The profiles for $j$ are similar to that for $\rho$, and they are omitted to save space.

\begin{figure}[ht]
\centering
\includegraphics[totalheight=2.0in]{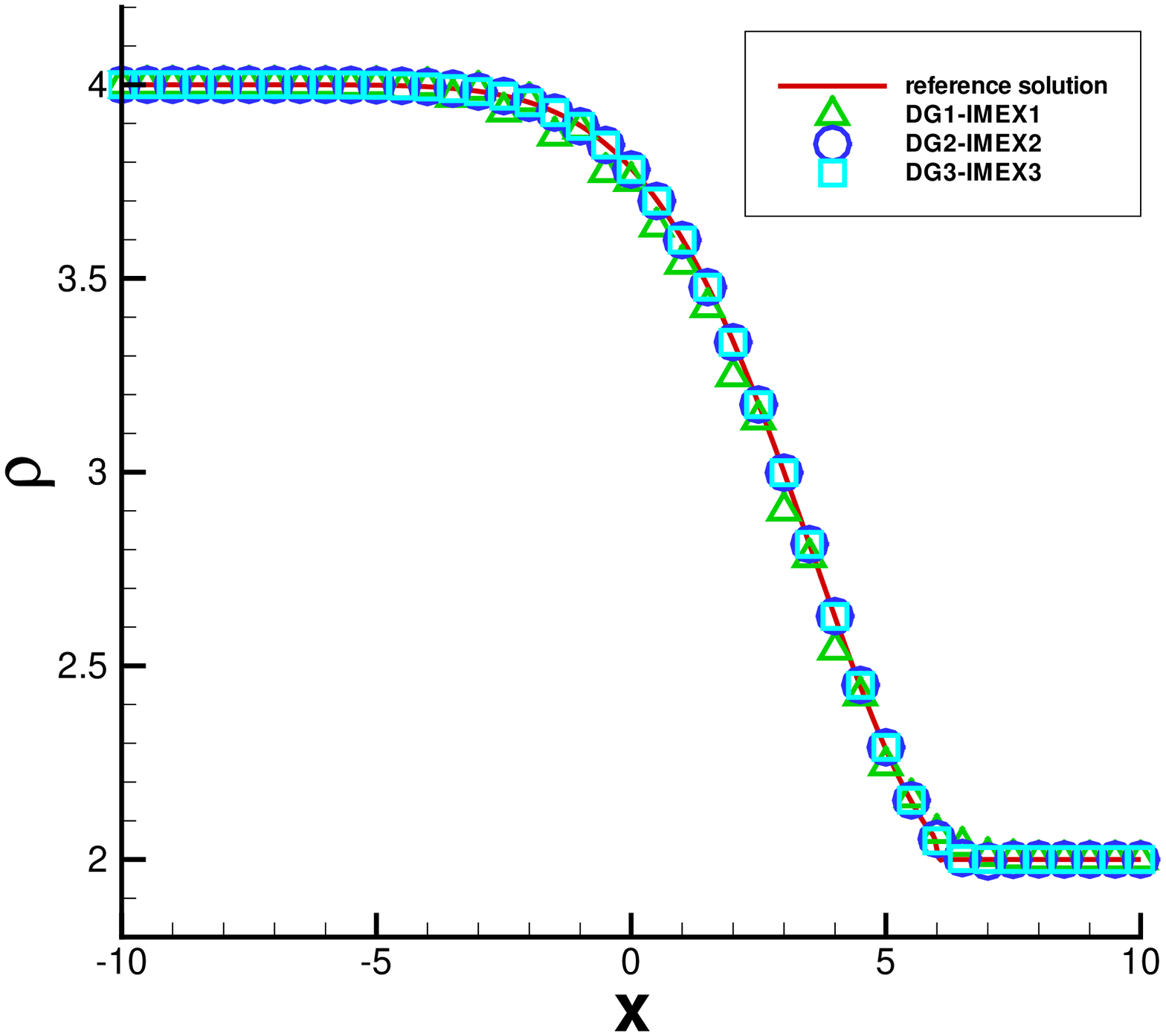},
\includegraphics[totalheight=2.0in]{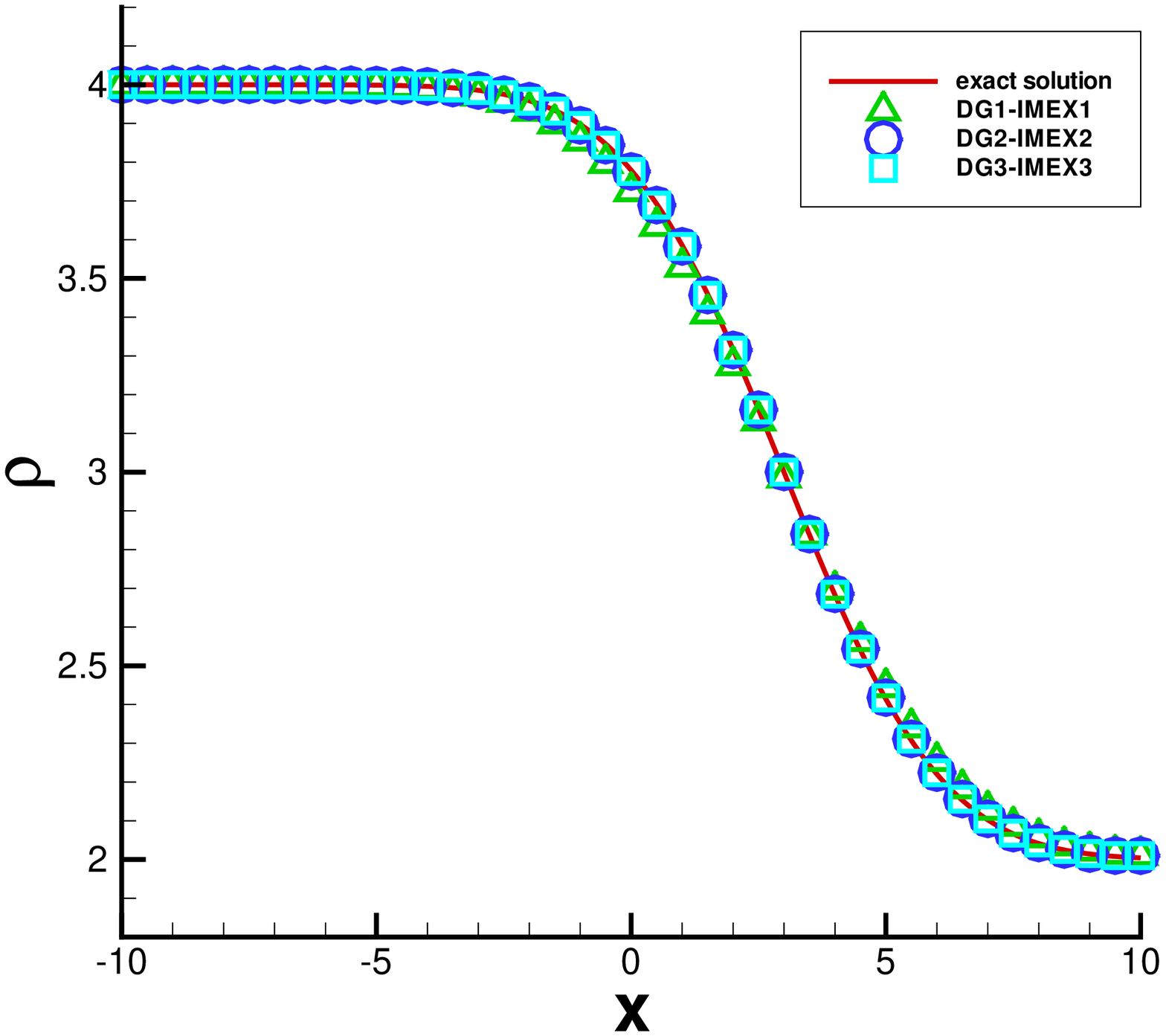}
\caption{Numerical solutions of the advection-diffusion equation with initial conditions (\ref{ad2}) by DG(k+1)-IMEX(k+1), $k=0, 1, 2$ at $T=3.0$. Left-right flux and $\Delta x=0.5$.
Left: the rarefied regime $\varepsilon=0.5$; Right: the parabolic regime $\varepsilon=10^{-6}$.}
\label{fig2}
\end{figure}

\subsection{Viscous Burgers' equation}
\label{sec:4.3}

Consider (\ref{C:B1}) with $C=1/2$.
The initial conditions are chosen to be two local Maxwellian
\begin{eqnarray}
 \begin{cases}
  \rho_L=2.0,  \qquad -10<x<0,\\
  \rho_R=1.0,  \qquad 0<x<10,
 \end{cases}
 \label{burger1}
\end{eqnarray}
with $j=\rho^2/(1+\sqrt{1+\rho^2\varepsilon^2})$. For this problem, we show in Figure \ref{fig3} the numerical solutions $\rho$ and $j$ of DG(k+1)-IMEX(k+1), $k=0, 1, 2$ using the left-right flux in the rarefied regime ($\varepsilon=0.4$) and in the parabolic regime ($\varepsilon=10^{-6}$) with
$\Delta x=0.25$, which are compared with the reference solutions obtained by DG3-IMEX3 with $\Delta x=0.04$. Numerical solutions are in very good agreement with the reference solutions.

\begin{figure}[ht]
\centering
\includegraphics[totalheight=2.0in]{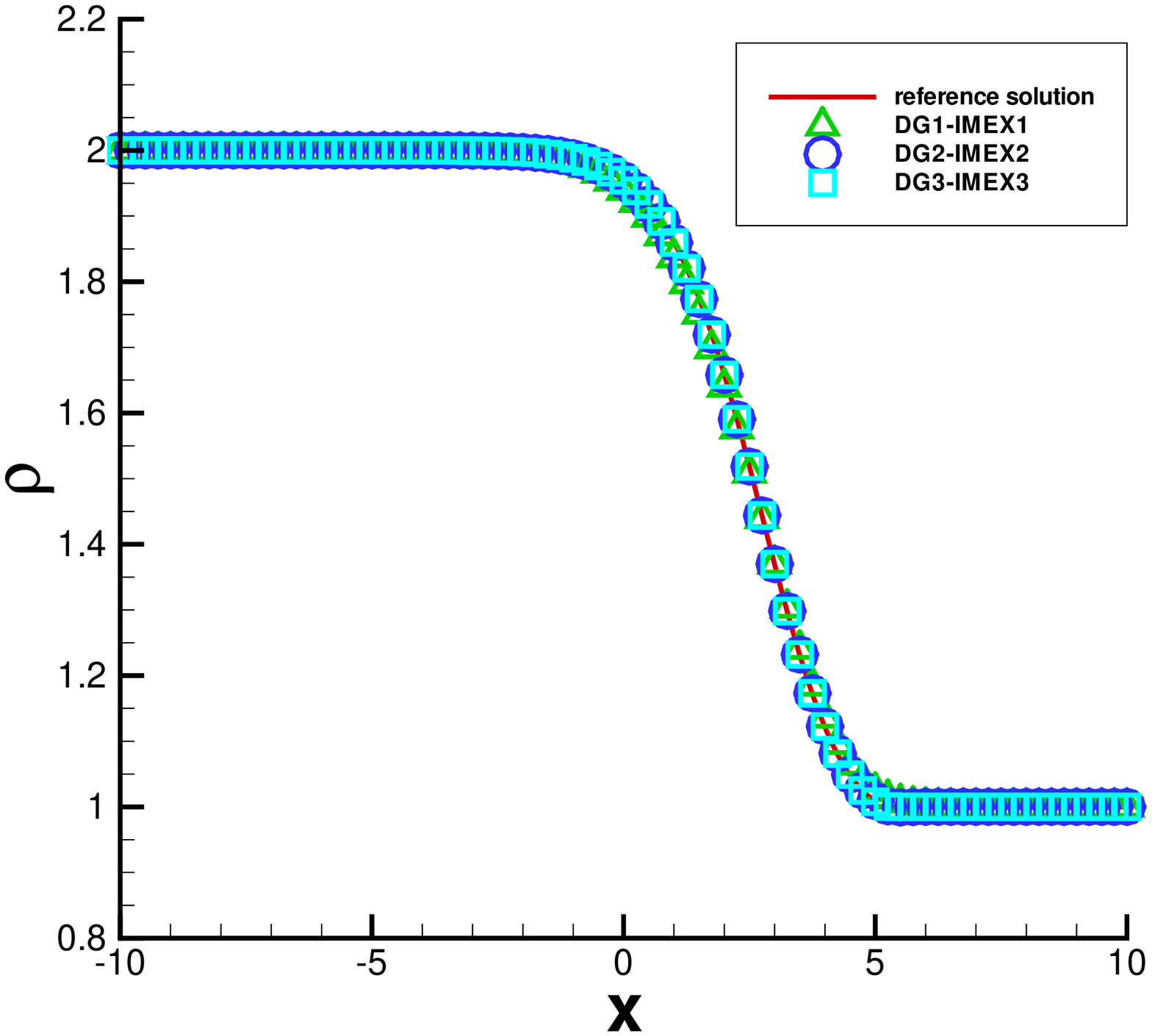},
\includegraphics[totalheight=2.0in]{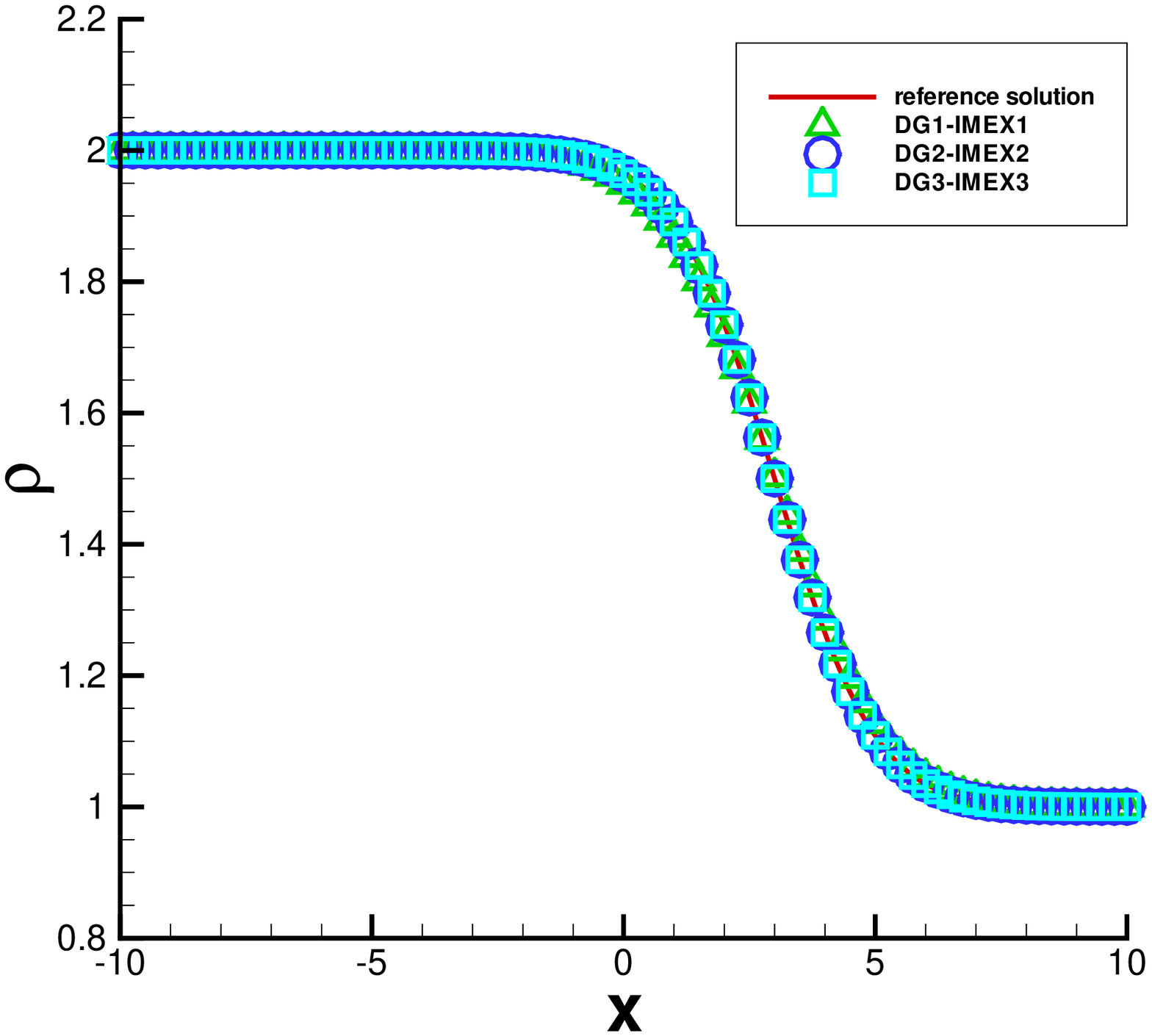}\\
\includegraphics[totalheight=2.0in]{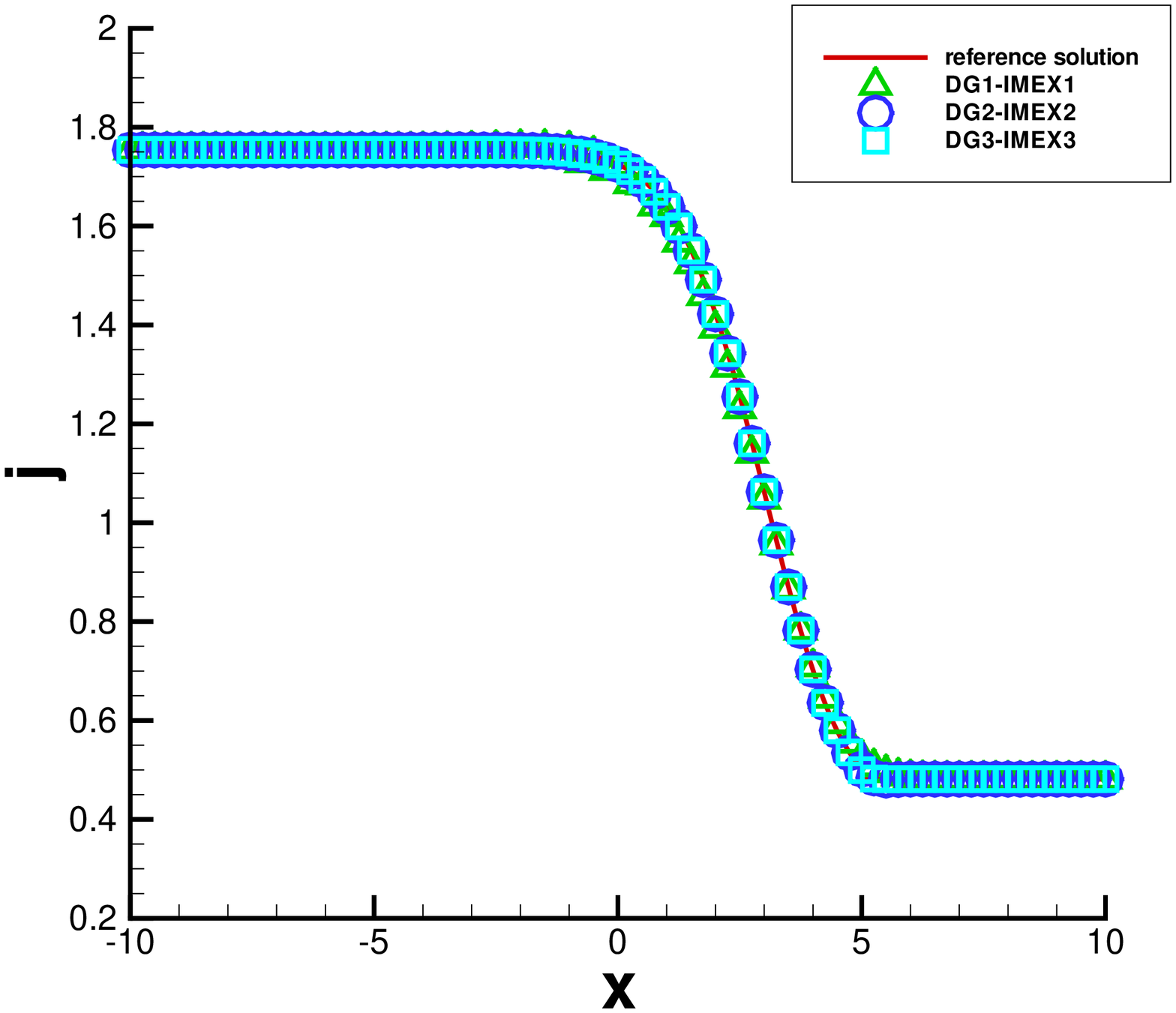},
\includegraphics[totalheight=2.0in]{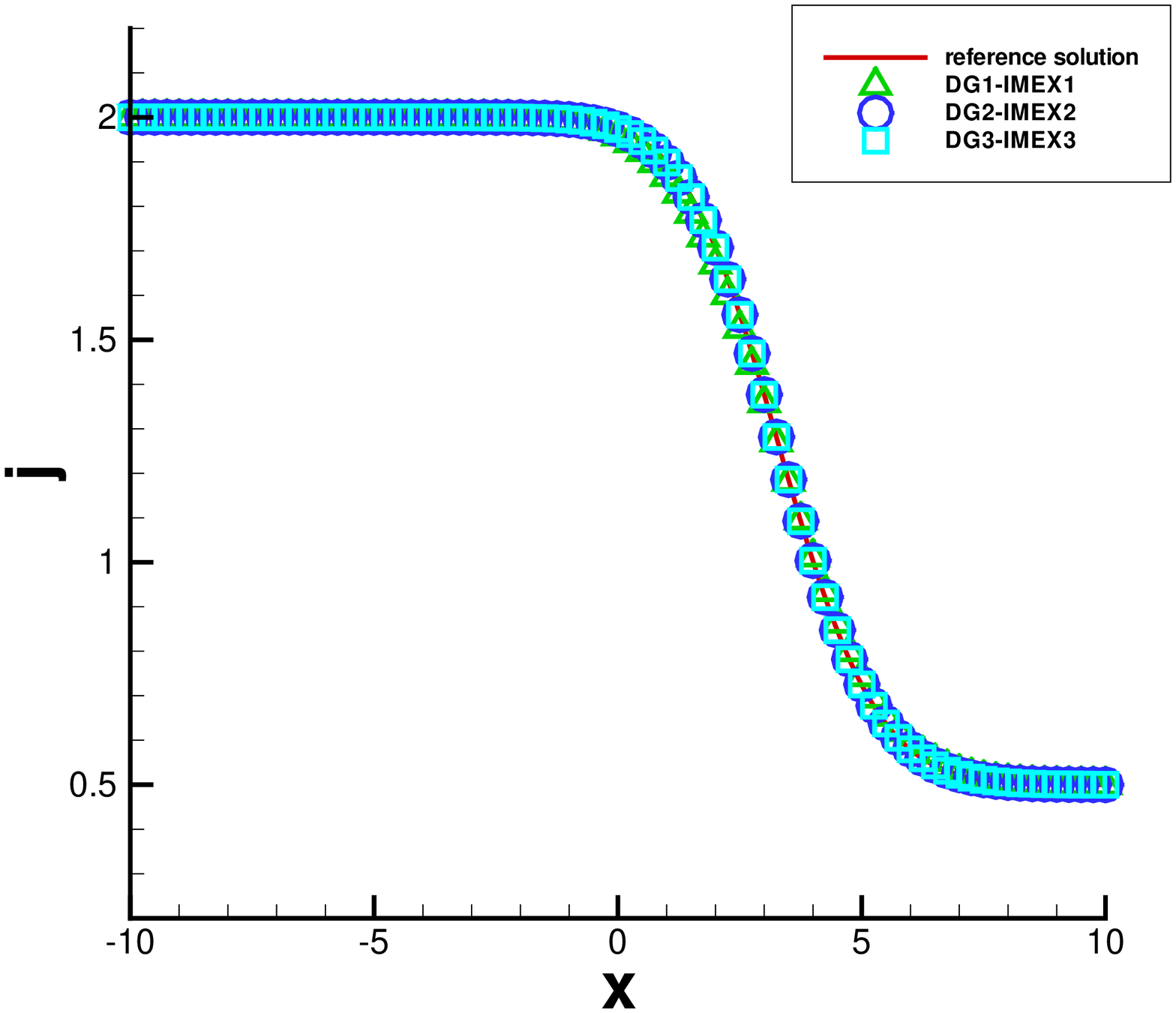}\\
\caption{Numerical solutions of Burgers' equation with initial conditions (\ref{burger1}) by DG(k+1)-IMEX(k+1), $k=0, 1, 2$ at $T=2.0$. Left-right flux and $\Delta x=0.25$.
Left: the rarefied regime $\varepsilon=0.4$; Right: the parabolic regime
$\varepsilon=10^{-6}$. Symbol: numerical solution; Solid line: reference solution.}
\label{fig3}
\end{figure}

For this nonlinear example, an exact smooth shock wave solution was given by Ruijgrok and Wu in \cite{ruijgrok1982completely}.
We choose the similar boundary condition as (\ref{burger1})
\begin{eqnarray}
\lim_{x\rightarrow +\infty}\rho= \rho^+=2.0,  \qquad \lim_{x\rightarrow -\infty}\rho= \rho^-=1.0,
 \label{burger2}
\end{eqnarray}
with $j^\pm=(\rho^\pm)^2/(1+\sqrt{1+(\rho^\pm)^2\varepsilon^2})$. Let
\begin{eqnarray}
u^\pm=\rho^\pm+\varepsilon j^\pm, \quad v^\pm=\rho^\pm-\varepsilon j^\pm.
\end{eqnarray}
The distribution function $u(x,t)$ and $v(x,t)$ are given by
\begin{eqnarray*}
u(x,t)=\frac{u^++u^-e^{-(\xi-\xi_0)/X_0}}{1+e^{-(\xi-\xi_0)/X_0}}, \quad
v(x,t)=\frac{v^++v^-e^{-(\xi-\xi_0)/X_0}}{1+e^{-(\xi-\xi_0)/X_0}}
\end{eqnarray*}
where $\xi=(x-w t/\varepsilon)/2$ and let $\xi_0=0$. The width $X_0$ and the velocity $w$ of the shock wave
are
\begin{eqnarray*}
X_0=\frac{1+w}{u^--u^+}=\frac{1-w}{v^--v^+}, \quad w=\frac{u^--u^+-v^-+v^+}{u^--u^++v^--v^+}.
\end{eqnarray*}
With above, the exact solutions $\rho(x,t)$ and $j(x,t)$ are given by
\begin{eqnarray}
\rho(x,t)=\frac{u(x,t)+v(x,t)}{2}, \quad j(x,t)=\frac{u(x,t)-v(x,t)}{2\varepsilon}
\label{burger2end}
\end{eqnarray}
 The computational domain is taken to be $[-40, 40]$ with inflow and outflow boundary conditions. Starting from the smooth initial condition at $t=0$, we compute up to $T=1$ for $\varepsilon=0.5, 10^{-2}, 10^{-6}$.  The numerical results by  DG(k+1)-IMEX(k+1) with the left-right alternating flux are presented in Tables \ref{tab31}-\ref{tab33} for  $k=0, 1, 2$, respectively, and $(k+1)^{th}$ order is observed.
  In Tables \ref{tab34}-\ref{tab36},  $(k+1)^{th}$ order is observed for even $k$ and $k^{th}$ order for odd $k$ when DG(k+1)-IMEX(k+1) are combined with the central flux \eqref{eq:flux:2}.



\begin{table}
\centering
\caption{$L^1$ errors and orders of $\rho$ and $j$ for Burgers' equation with the exact solution (\ref{burger2})-(\ref{burger2end}), $T=1.0$, DG1-IMEX1 with left-right flux.}
\vspace{0.2cm}
  \begin{tabular}{|c|c|c|c|c|c|}
    \hline
    &N  &  $L^1$ error of $\rho$ & order   & $L^1$ error of $j$  & order \\\hline
\multirow{5}{*}{$\varepsilon=0.5$}&    10 &     2.98E-02 &       --&     3.46E-02 &       --  \\  \cline{2-6}
   & 20 &     1.39E-02 &     1.11&     1.57E-02 &     1.14  \\  \cline{2-6}
   & 40 &     6.48E-03 &     1.10&     7.92E-03 &     0.98  \\  \cline{2-6}
   & 80 &     3.25E-03 &     0.99&     3.93E-03 &     1.01  \\  \cline{2-6}
   &160 &     1.66E-03 &     0.97&     1.89E-03 &     1.06  \\  \hline
\multirow{5}{*}{$\varepsilon=10^{-2}$}&    10 &     3.08E-02 &       --&     4.52E-02 &       --  \\  \cline{2-6}
   & 20 &     1.41E-02 &     1.13&     2.00E-02 &     1.18  \\  \cline{2-6}
   & 40 &     6.38E-03 &     1.14&     9.13E-03 &     1.13  \\  \cline{2-6}
   & 80 &     3.39E-03 &     0.91&     4.68E-03 &     0.96  \\  \cline{2-6}
   &160 &     1.76E-03 &     0.94&     2.42E-03 &     0.95  \\  \hline
\multirow{5}{*}{$\varepsilon=10^{-6}$}&    10 &     3.08E-02 &       --&     4.52E-02 &       --  \\  \cline{2-6}
   & 20 &     1.41E-02 &     1.13&     2.00E-02 &     1.18  \\  \cline{2-6}
   & 40 &     6.38E-03 &     1.14&     9.14E-03 &     1.13  \\  \cline{2-6}
   & 80 &     3.39E-03 &     0.91&     4.68E-03 &     0.97  \\  \cline{2-6}
   &160 &     1.76E-03 &     0.94&     2.42E-03 &     0.95  \\  \hline
  \end{tabular}
\label{tab31}
\end{table}

\begin{table}
\centering
\caption{$L^1$ errors and orders of $\rho$ and $j$ for Burgers' equation with the exact solution (\ref{burger2})-(\ref{burger2end}), $T=1.0$, DG2-IMEX2 with left-right flux.}
\vspace{0.2cm}
  \begin{tabular}{|c|c|c|c|c|c|}
    \hline
    &N  &  $L^1$ error of $\rho$ & order   & $L^1$ error of $j$  & order \\\hline
\multirow{5}{*}{$\varepsilon=0.5$}&    10 &     5.78E-03 &       --&     7.07E-03 &       --  \\  \cline{2-6}
   & 20 &     1.84E-03 &     1.65&     2.48E-03 &     1.51  \\  \cline{2-6}
   & 40 &     4.27E-04 &     2.10&     6.78E-04 &     1.87  \\  \cline{2-6}
   & 80 &     9.80E-05 &     2.13&     1.50E-04 &     2.18  \\  \cline{2-6}
   &160 &     2.50E-05 &     1.97&     4.72E-05 &     1.67  \\  \hline
\multirow{5}{*}{$\varepsilon=10^{-2}$}&    10 &     6.26E-03 &       --&     9.20E-03 &       --  \\  \cline{2-6}
   & 20 &     1.85E-03 &     1.76&     3.57E-03 &     1.36  \\  \cline{2-6}
   & 40 &     4.27E-04 &     2.11&     9.83E-04 &     1.86  \\  \cline{2-6}
   & 80 &     1.22E-04 &     1.80&     2.43E-04 &     2.02  \\  \cline{2-6}
   &160 &     3.34E-05 &     1.87&     6.03E-05 &     2.01  \\  \hline
\multirow{5}{*}{$\varepsilon=10^{-6}$}&    10 &     6.24E-03 &       --&     9.25E-03 &       --  \\  \cline{2-6}
   & 20 &     1.85E-03 &     1.76&     3.62E-03 &     1.35  \\  \cline{2-6}
   & 40 &     4.29E-04 &     2.11&     9.84E-04 &     1.88  \\  \cline{2-6}
   & 80 &     1.23E-04 &     1.80&     2.43E-04 &     2.02  \\  \cline{2-6}
   &160 &     3.37E-05 &     1.87&     6.04E-05 &     2.01  \\  \hline
  \end{tabular}
\label{tab32}
\end{table}

\begin{table}
\centering
\caption{$L^1$ errors and orders of $\rho$ and $j$ for Burgers' equation with the exact solution (\ref{burger2})-(\ref{burger2end}), $T=1.0$, DG3-IMEX3 with left-right flux.}
\vspace{0.2cm}
  \begin{tabular}{|c|c|c|c|c|c|}
    \hline
    &N  &  $L^1$ error of $\rho$ & order   & $L^1$ error of $j$  & order \\\hline
\multirow{5}{*}{$\varepsilon=0.5$}&    10 &     2.51E-03 &       --&     5.39E-03 &       --  \\  \cline{2-6}
   & 20 &     2.48E-04 &     3.34&     4.06E-04 &     3.73  \\  \cline{2-6}
   & 40 &     2.74E-05 &     3.18&     4.70E-05 &     3.11  \\  \cline{2-6}
   & 80 &     3.54E-06 &     2.96&     6.11E-06 &     2.94  \\  \cline{2-6}
   &160 &     4.56E-07 &     2.96&     7.66E-07 &     3.00  \\  \hline
\multirow{5}{*}{$\varepsilon=10^{-2}$}&    10 &     2.13E-03 &       --&     3.92E-03 &       --  \\  \cline{2-6}
   & 20 &     2.26E-04 &     3.24&     4.26E-04 &     3.20  \\  \cline{2-6}
   & 40 &     3.30E-05 &     2.77&     5.93E-05 &     2.84  \\  \cline{2-6}
   & 80 &     4.71E-06 &     2.81&     7.66E-06 &     2.95  \\  \cline{2-6}
   &160 &     6.23E-07 &     2.92&     9.68E-07 &     2.98  \\  \hline
\multirow{5}{*}{$\varepsilon=10^{-6}$}&    10 &     2.13E-03 &       --&     4.01E-03 &       --  \\  \cline{2-6}
   & 20 &     2.26E-04 &     3.24&     4.27E-04 &     3.23  \\  \cline{2-6}
   & 40 &     3.33E-05 &     2.76&     5.97E-05 &     2.84  \\  \cline{2-6}
   & 80 &     4.76E-06 &     2.81&     7.78E-06 &     2.94  \\  \cline{2-6}
   &160 &     6.29E-07 &     2.92&     9.94E-07 &     2.97  \\  \hline
  \end{tabular}
\label{tab33}
\end{table}


\begin{table}
\centering
\caption{$L^1$ errors and orders of $\rho$ and $j$ for Burgers' equation with the exact solution (\ref{burger2})-(\ref{burger2end}), $T=1.0$, DG1-IMEX1 with central flux.}
\vspace{0.2cm}
  \begin{tabular}{|c|c|c|c|c|c|}
    \hline
    &N  &  $L^1$ error of $\rho$ & order   & $L^1$ error of $j$  & order \\\hline
\multirow{5}{*}{$\varepsilon=0.5$}&    10 &     2.89E-02 &       --&     3.71E-02 &       --  \\  \cline{2-6}
   & 20 &     1.30E-02 &     1.16&     1.61E-02 &     1.20  \\  \cline{2-6}
   & 40 &     6.48E-03 &     1.00&     8.51E-03 &     0.92  \\  \cline{2-6}
   & 80 &     3.18E-03 &     1.03&     4.06E-03 &     1.07  \\  \cline{2-6}
   &160 &     1.55E-03 &     1.03&     1.93E-03 &     1.08  \\  \hline
\multirow{5}{*}{$\varepsilon=10^{-2}$}&    10 &     3.02E-02 &       --&     5.08E-02 &       --  \\  \cline{2-6}
   & 20 &     1.33E-02 &     1.18&     2.17E-02 &     1.23  \\  \cline{2-6}
   & 40 &     7.07E-03 &     0.91&     1.11E-02 &     0.97  \\  \cline{2-6}
   & 80 &     3.18E-03 &     1.15&     4.80E-03 &     1.21  \\  \cline{2-6}
   &160 &     1.53E-03 &     1.05&     2.31E-03 &     1.06  \\  \hline
\multirow{5}{*}{$\varepsilon=10^{-6}$}&    10 &     3.02E-02 &       --&     5.08E-02 &       --  \\  \cline{2-6}
   & 20 &     1.33E-02 &     1.18&     2.17E-02 &     1.23  \\  \cline{2-6}
   & 40 &     7.07E-03 &     0.91&     1.11E-02 &     0.97  \\  \cline{2-6}
   & 80 &     3.18E-03 &     1.15&     4.81E-03 &     1.21  \\  \cline{2-6}
   &160 &     1.53E-03 &     1.05&     2.31E-03 &     1.06  \\  \hline
  \end{tabular}
\label{tab34}
\end{table}

\begin{table}
\centering
\caption{$L^1$ errors and orders of $\rho$ and $j$ for Burgers' equation with the exact solution (\ref{burger2})-(\ref{burger2end}), $T=1.0$, DG2-IMEX2 with central flux.}
\vspace{0.2cm}
  \begin{tabular}{|c|c|c|c|c|c|}
    \hline
    &N  &  $L^1$ error of $\rho$ & order   & $L^1$ error of $j$  & order \\ \hline
\multirow{5}{*}{$\varepsilon=0.5$}&    10 &     6.98E-03 &       --&     7.74E-03 &       --  \\  \cline{2-6}
   & 20 &     3.46E-03 &     1.01&     3.33E-03 &     1.21  \\  \cline{2-6}
   & 40 &     1.66E-03 &     1.06&     1.25E-03 &     1.41  \\  \cline{2-6}
   & 80 &     8.23E-04 &     1.01&     4.48E-04 &     1.48  \\  \cline{2-6}
   &160 &     4.17E-04 &     0.98&     1.43E-04 &     1.65  \\  \hline
\multirow{5}{*}{$\varepsilon=10^{-2}$}&    10 &     8.25E-03 &       --&     1.50E-02 &       --  \\  \cline{2-6}
   & 20 &     3.65E-03 &     1.18&     7.21E-03 &     1.06  \\  \cline{2-6}
   & 40 &     1.74E-03 &     1.07&     3.49E-03 &     1.05  \\  \cline{2-6}
   & 80 &     8.50E-04 &     1.04&     1.68E-03 &     1.06  \\  \cline{2-6}
   &160 &     4.23E-04 &     1.01&     7.96E-04 &     1.07  \\  \hline
\multirow{5}{*}{$\varepsilon=10^{-6}$}&    10 &     8.23E-03 &       --&     1.51E-02 &       --  \\  \cline{2-6}
   & 20 &     3.65E-03 &     1.17&     7.29E-03 &     1.05  \\  \cline{2-6}
   & 40 &     1.74E-03 &     1.07&     3.57E-03 &     1.03  \\  \cline{2-6}
   & 80 &     8.42E-04 &     1.05&     1.75E-03 &     1.03  \\  \cline{2-6}
   &160 &     4.14E-04 &     1.02&     8.63E-04 &     1.02  \\  \hline
  \end{tabular}
\label{tab35}
\end{table}

\begin{table}
\centering
\caption{$L^1$ errors and orders of $\rho$ and $j$ for Burgers' equation with the exact solution (\ref{burger2})-(\ref{burger2end}), $T=1.0$, DG3-IMEX3 with central flux.}
\vspace{0.2cm}
  \begin{tabular}{|c|c|c|c|c|c|}
    \hline
    &N  &  $L^1$ error of $\rho$ & order   & $L^1$ error of $j$  & order \\\hline
\multirow{5}{*}{$\varepsilon=0.5$}&    10 &     3.98E-03 &       --&     4.80E-03 &       --  \\  \cline{2-6}
   & 20 &     5.97E-04 &     2.74&     6.66E-04 &     2.85  \\  \cline{2-6}
   & 40 &     4.11E-05 &     3.86&     4.20E-05 &     3.99  \\  \cline{2-6}
   & 80 &     4.20E-06 &     3.29&     4.29E-06 &     3.29  \\  \cline{2-6}
   &160 &     5.08E-07 &     3.05&     4.95E-07 &     3.12  \\  \hline
\multirow{5}{*}{$\varepsilon=10^{-2}$}&    10 &     4.92E-03 &       --&     1.05E-02 &       --  \\  \cline{2-6}
   & 20 &     7.21E-04 &     2.77&     1.52E-03 &     2.79  \\  \cline{2-6}
   & 40 &     3.85E-05 &     4.23&     6.54E-05 &     4.54  \\  \cline{2-6}
   & 80 &     3.55E-06 &     3.44&     5.34E-06 &     3.61  \\  \cline{2-6}
   &160 &     4.20E-07 &     3.08&     6.29E-07 &     3.08  \\  \hline
\multirow{5}{*}{$\varepsilon=10^{-6}$}&    10 &     4.95E-03 &       --&     1.06E-02 &       --  \\  \cline{2-6}
   & 20 &     7.26E-04 &     2.77&     1.56E-03 &     2.77  \\  \cline{2-6}
   & 40 &     3.86E-05 &     4.23&     6.64E-05 &     4.55  \\  \cline{2-6}
   & 80 &     3.55E-06 &     3.44&     5.49E-06 &     3.60  \\  \cline{2-6}
   &160 &     4.20E-07 &     3.08&     6.56E-07 &     3.06  \\  \hline
  \end{tabular}
\label{tab36}
\end{table}

\subsection{Porous media equation}
\label{sec:4.4}

Consider (\ref{C:porus1}) with $m=-1$ and $K=1/2$, and the limiting equation is the porous media equation. We compare the numerical solutions of DG(k+1)-IMEX(k+1), $k=0, 1, 2$ for $\varepsilon=10^{-6}$ with the exact Barenblatt solution of the limiting porous media equation,
\begin{eqnarray}
\begin{cases}
 \rho(x,t)=\frac{1}{R(t)}\left[1-\left(\frac{x}{R(t)}\right)^2\right], \quad j(x,t)=\rho(x,t)\frac{4x}{R(t)^3}, \quad |x|<R(t), \\
 \rho(x,t)=0, \quad j(x,t)=0, \quad |x|>R(t),
 \end{cases}
 \label{porous}
\end{eqnarray}
where $R(t)=[12(t+1)]^{1/3}$, $t\ge0$.

For this example, there are two implementation issues one needs to pay attention to. First, to avoid being divided by zero in the collision term (note $\rho$ is or is close to $0$ in part of the computational domain), the DG scheme is implemented in the nodal fashion \cite{hesthaven2008nodal}. Specifically, we use the Lagrangian basis functions at $k+1$ Gaussian points to represent a polynomial space of degree $k$.
Secondly, when alternating fluxes are used, to ensure the interface of $\rho=0$ propagating outward in time, we choose the right-left flux in the left half of the domain to ensure the left interface of $\rho = 0$ is propagating to the left, and the left-right flux in the right half of the domain to ensure the right interface of $\rho=0$ is propagating to the right. 
With this strategy, in the transition interval which uses the left-right flux as its flux on the left boundary, and right-left flux as its flux on the right boundary, the scheme loses one order accuracy. Because of this, the first order scheme becomes inconsistent and the corresponding result is not presented here.
In  Figure~\ref{fig4}, we show numerical results of DG(k+1)-IMEX(k+1) at $T=3.0$ with $\Delta x=0.5$:  with the left ones using  alternating fluxes for $k=1, 2$, and the right ones using  the central flux for $k=0,1,2$. The numerical solutions very well capture the exact solution of the limiting equation. Moreover, higher order schemes (i.e. DG3-IMEX3) demonstrate better resolution.

\begin{figure}[ht]
\centering
\includegraphics[totalheight=2.0in]{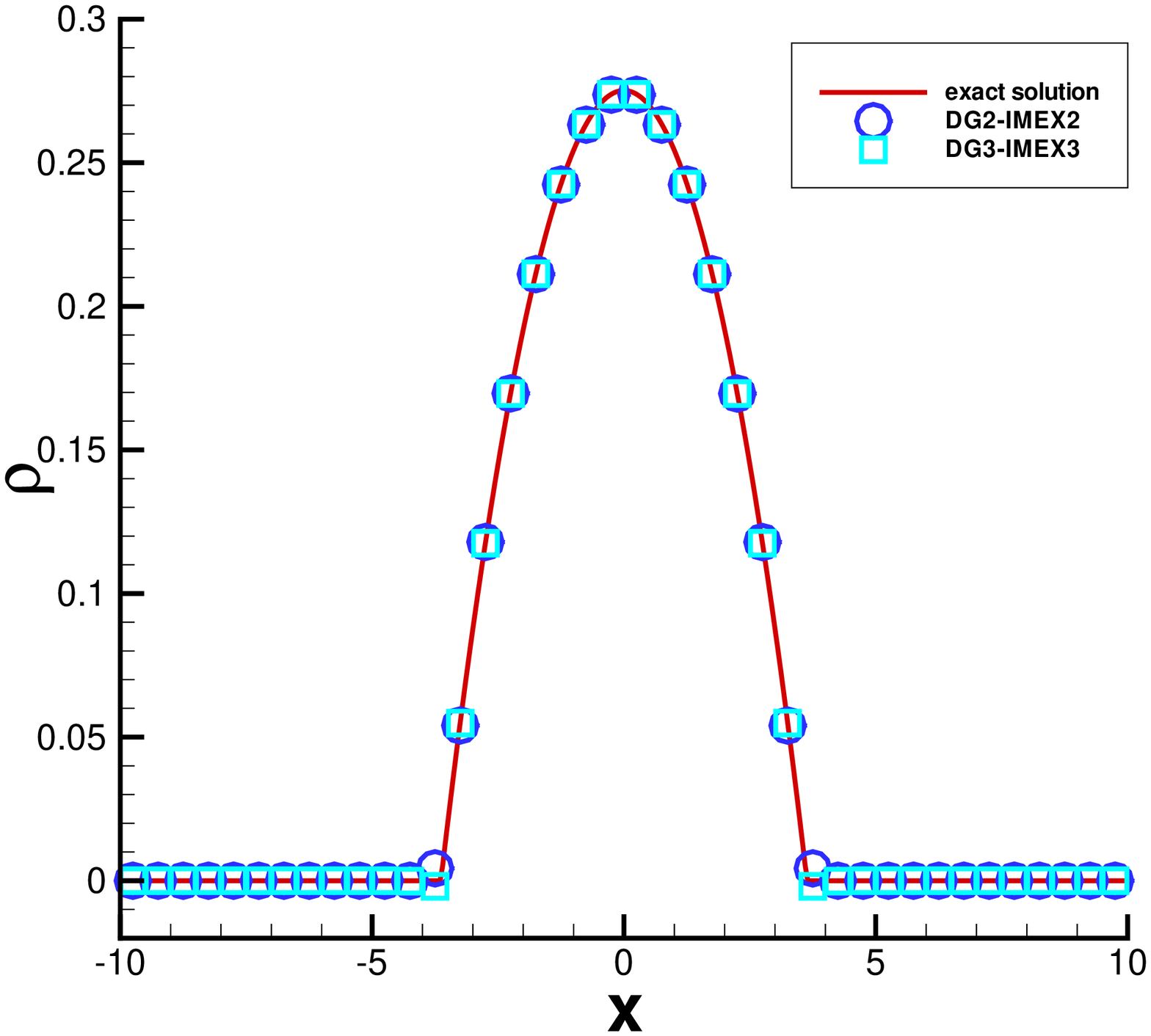},
\includegraphics[totalheight=2.0in]{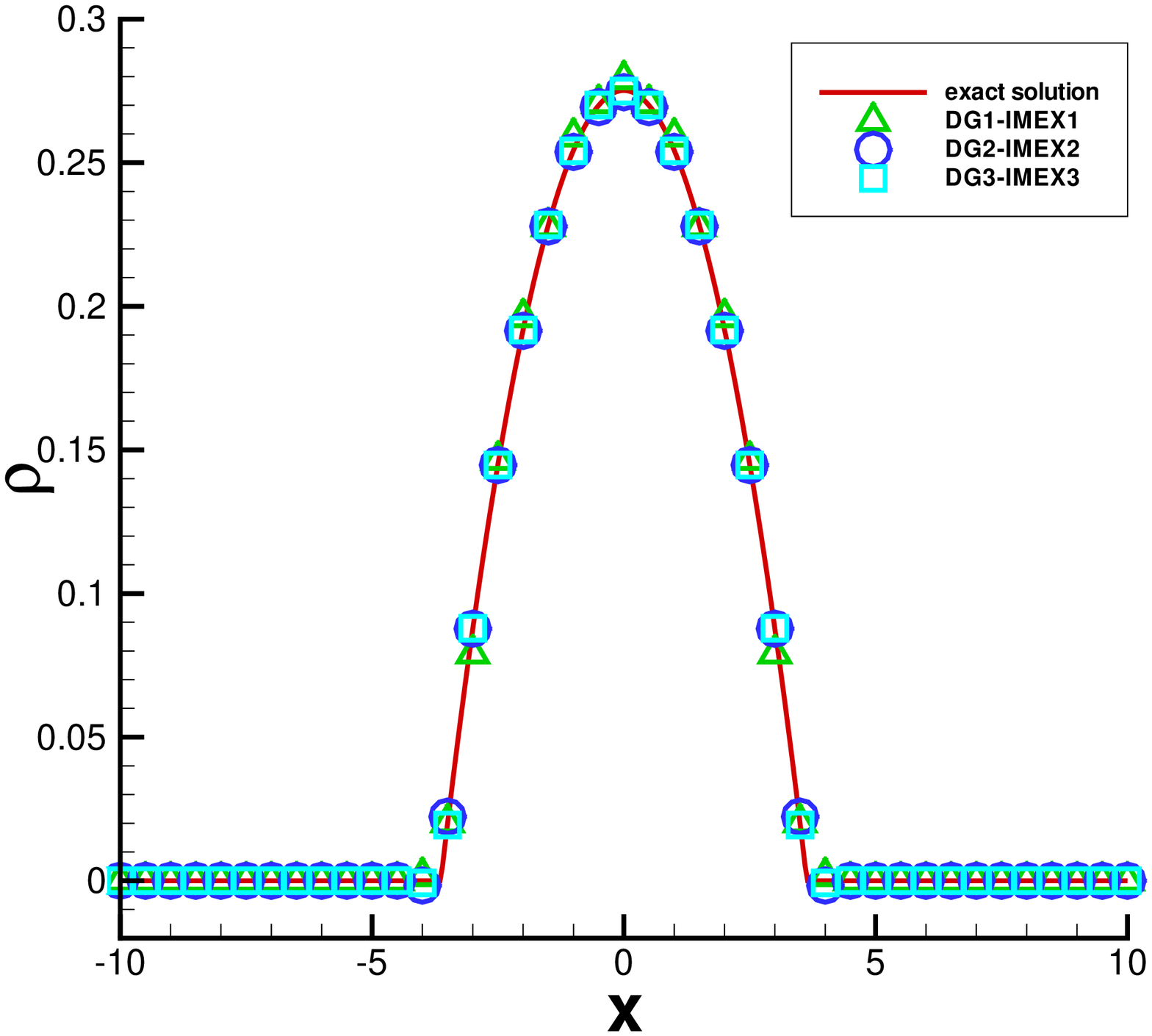}\\
\includegraphics[totalheight=2.0in]{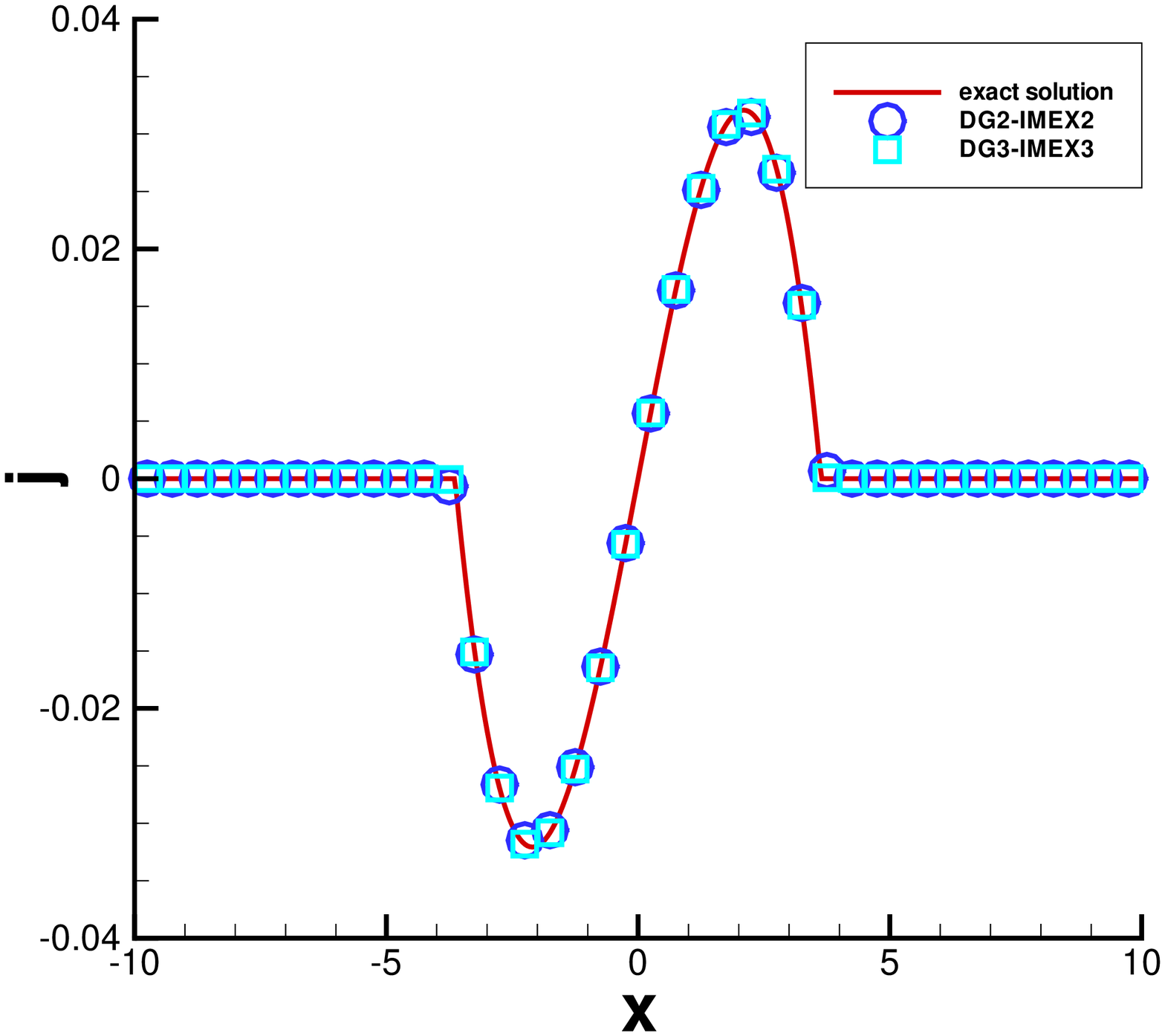},
\includegraphics[totalheight=2.0in]{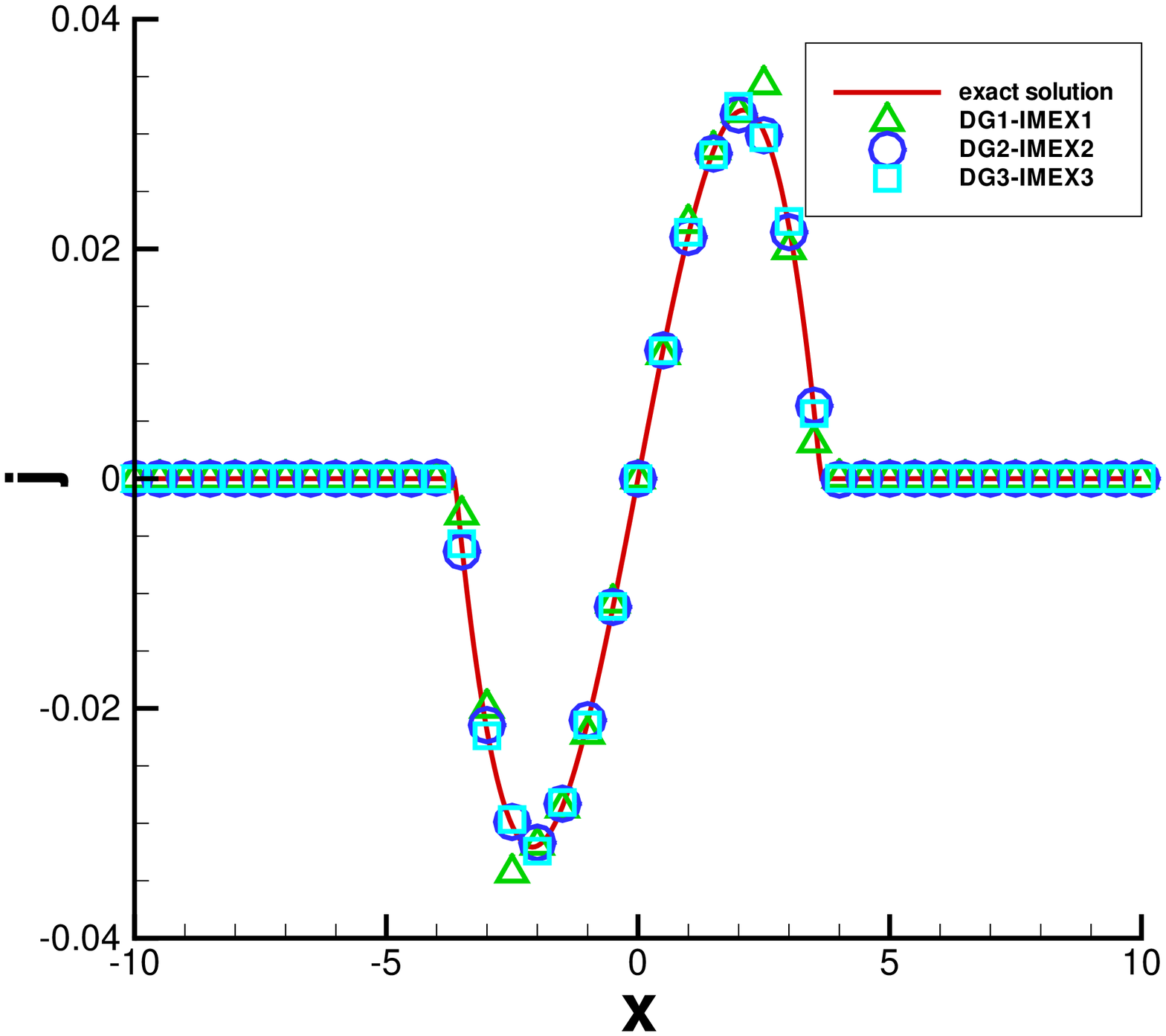}
\caption{Numerical solutions of porous media equation compared with the limiting Barenblatt solution (\ref{porous}). The parabolic regime $\varepsilon=10^{-6}$ at $T=3.0$ with $\Delta x=0.5$. Left: alternating flux; Right: central flux.}
\label{fig4}
\end{figure}

%% file: conclusion.tex
\section{Conclusion}
\label{sec6}
\setcounter{equation}{0}
\setcounter{figure}{0}
\setcounter{table}{0}

This paper is an initial effort in developing high order asymptotic preserving schemes for kinetic equations in different scalings.
For some discrete-velocity kinetic models in a diffusive scaling, we propose to employ arbitrarily high order DG spatial discretizations coupled with high order globally stiffly accurate IMEX scheme for an equivalent micro-macro decomposition of the kinetic equations.
The proposed schemes are asymptotic preserving in the sense of \cite{jin2010asymptotic}. Moreover, the proposed schemes are of high order in both space and time when $\varepsilon$ is of order $1$ and in the limit of $\varepsilon \rightarrow 0$. Some uniform stability analysis and error estimates for the proposed schemes will be rigorously established in a separate paper \cite{JLQX_analysis}.
Extensions to more general kinetic equations with different asymptotic limits will be explored in the future. 

%% file: acknowledge.tex
\bigskip
\noindent
{\bf Acknowledgement.}
This project was initiated during the authors' participation at the ICERM Semester Program on ``Kinetic Theory and Computation" in the fall of 2011. The authors want to thank for the generous support from the Institute. Part of the work was done at MFO in Oberwolfach during a Research in Pairs program. The first three authors appreciate the support and hospitality of the Institute.


%% file: paper_AP.bbl
\begin{thebibliography}{10}

\bibitem{ascher1997implicit}
{\sc U.~Ascher, S.~Ruuth, and R.~Spiteri}, {\em Implicit-explicit runge-kutta
  methods for time-dependent partial differential equations}, Applied Numerical
  Mathematics, 25 (1997), pp.~151--167.

\bibitem{bardos1991fluid}
{\sc C.~Bardos, F.~Golse, and D.~Levermore}, {\em Fluid dynamic limits of
  kinetic equations. i. formal derivations}, Journal of statistical physics, 63
  (1991), pp.~323--344.

\bibitem{boscarino2008error}
{\sc S.~Boscarino}, {\em Error analysis of imex runge-kutta methods derived
  from differential-algebraic systems}, SIAM Journal on Numerical Analysis, 45
  (2008), pp.~1600--1621.

\bibitem{boscarino2011implicit}
{\sc S.~Boscarino, L.~Pareschi, and G.~Russo}, {\em Implicit-explicit
  runge--kutta schemes for hyperbolic systems and kinetic equations in the
  diffusion limit}, SIAM Journal on Scientific Computing, 35 (2013),
  pp.~A22--A51.

\bibitem{boscarino2010class}
{\sc S.~Boscarino and G.~Russo}, {\em On a class of uniformly accurate imex
  runge-kutta schemes and applications to hyperbolic systems with relaxation},
  SIAM Journal on Scientific Computing, 31 (2010), p.~1926.

\bibitem{boscarino2013flux}
{\sc S.~Boscarino and G.~Russo}, {\em Flux-explicit imex runge--kutta schemes
  for hyperbolic to parabolic relaxation problems}, SIAM Journal on Numerical
  Analysis, 51 (2013), pp.~163--190.

\bibitem{bourgat1994coupling}
{\sc J.~Bourgat, P.~Le~Tallec, B.~Perthame, and Y.~Qiu}, {\em Coupling
  boltzmann and euler equations without overlapping}, Contemporary Mathematics,
  157 (1994), pp.~377--377.

\bibitem{carrillo2008numerical}
{\sc J.~A. Carrillo, T.~Goudon, P.~Lafitte, and F.~Vecil}, {\em Numerical
  schemes of diffusion asymptotics and moment closures for kinetic equations},
  Journal of Scientific Computing, 36 (2008), pp.~113--149.

\bibitem{cockburn1998local}
{\sc B.~Cockburn and C.-W. Shu}, {\em The local discontinuous galerkin method
  for time-dependent convection-diffusion systems}, SIAM Journal on Numerical
  Analysis, 35 (1998), pp.~2440--2463.

\bibitem{cockburn2001runge}
\leavevmode\vrule height 2pt depth -1.6pt width 23pt, {\em {Runge--Kutta
  discontinuous Galerkin methods for convection-dominated problems}}, Journal
  of Scientific Computing, 16 (2001), pp.~173--261.

\bibitem{coulombel2005diffusion}
{\sc J.-F. Coulombel, F.~Golse, and T.~Goudon}, {\em Diffusion approximation
  and entropy-based moment closure for kinetic equations}, Asymptotic Analysis,
  45 (2005), pp.~1--39.

\bibitem{degond2005smooth}
{\sc P.~Degond and S.~Jin}, {\em A smooth transition model between kinetic and
  diffusion equations}, SIAM journal on numerical analysis, 42 (2005),
  pp.~2671--2687.

\bibitem{guermond2010asymptotic}
{\sc J.-L. Guermond and G.~Kanschat}, {\em Asymptotic analysis of upwind
  discontinuous galerkin approximation of the radiative transport equation in
  the diffusive limit}, SIAM Journal on Numerical Analysis, 48 (2010),
  pp.~53--78.

\bibitem{hesthaven2008nodal}
{\sc J.~S. Hesthaven and T.~Warburton}, {\em Nodal discontinuous Galerkin
  methods: algorithms, analysis, and applications}, vol.~54, Springerverlag New
  York, 2008.

\bibitem{JLQX_analysis}
{\sc J.~Jang, F.~Li, J.-M. Qiu, and T.~Xiong}, {\em Analysis of asymptotic
  preserving dg-imex schemes for discrete-velocity kinetic equations in a
  diffusive scaling}, in preparation.

\bibitem{JiangShu1994}
{\sc G.-S. Jiang, , and C.-W. Shu}, {\em {On a cell entropy inequality for
  discontinuous Galerkin methods}}, Mathematics of Computation, 62 (1994),
  pp.~531--538.

\bibitem{jin2010asymptotic}
{\sc S.~Jin}, {\em Asymptotic preserving (ap) schemes for multiscale kinetic
  and hyperbolic equations: a review}, Lecture Notes for Summer School on�
  Methods and Models of Kinetic Theory�(M\&MKT), Porto Ercole (Grosseto,
  Italy),  (2010).

\bibitem{jin1996numerical}
{\sc S.~Jin and C.~Levermore}, {\em Numerical schemes for hyperbolic
  conservation laws with stiff relaxation terms}, Journal of computational
  physics, 126 (1996), pp.~449--467.

\bibitem{jin1998diffusive}
{\sc S.~Jin, L.~Pareschi, and G.~Toscani}, {\em Diffusive relaxation schemes
  for multiscale discrete-velocity kinetic equations}, SIAM Journal on
  Numerical Analysis, 35 (1998), pp.~2405--2439.

\bibitem{jin2000uniformly}
\leavevmode\vrule height 2pt depth -1.6pt width 23pt, {\em Uniformly accurate
  diffusive relaxation schemes for multiscale transport equations}, SIAM
  Journal on Numerical Analysis, 38 (2000), pp.~913--936.

\bibitem{klar1998asymptotic}
{\sc A.~Klar}, {\em An asymptotic-induced scheme for nonstationary transport
  equations in the diffusive limit}, SIAM journal on numerical analysis, 35
  (1998), pp.~1073--1094.

\bibitem{lemou2010new}
{\sc M.~Lemou and L.~Mieussens}, {\em A new asymptotic preserving scheme based
  on micro-macro formulation for linear kinetic equations in the diffusion
  limit}, SIAM Journal on Scientific Computing, 31 (2010), pp.~334--368.

\bibitem{levermore1996moment}
{\sc C.~D. Levermore}, {\em Moment closure hierarchies for kinetic theories},
  Journal of Statistical Physics, 83 (1996), pp.~1021--1065.

\bibitem{liu2010analysis}
{\sc J.~Liu and L.~Mieussens}, {\em Analysis of an asymptotic preserving scheme
  for linear kinetic equations in the diffusion limit}, SIAM Journal on
  Numerical Analysis, 48 (2010), pp.~1474--1491.

\bibitem{liu2004boltzmann}
{\sc T.-P. Liu and S.-H. Yu}, {\em Boltzmann equation: micro-macro
  decompositions and positivity of shock profiles}, Communications in
  mathematical physics, 246 (2004), pp.~133--179.

\bibitem{naldi1998numerical}
{\sc G.~Naldi and L.~Pareschi}, {\em Numerical schemes for kinetic equations in
  diffusive regimes}, Applied mathematics letters, 11 (1998), pp.~29--35.

\bibitem{pareschi2005implicit}
{\sc L.~Pareschi and G.~Russo}, {\em Implicit--explicit runge--kutta schemes
  and applications to hyperbolic systems with relaxation}, Journal of
  Scientific computing, 25 (2005), pp.~129--155.

\bibitem{pareschi2011efficient}
\leavevmode\vrule height 2pt depth -1.6pt width 23pt, {\em Efficient asymptotic
  preserving deterministic methods for the boltzmann equation},  (2011).

\bibitem{riviere2008DG}
{\sc B.~Rivi\`ere}, {\em Discontinuous Galerkin methods for solving elliptic
  and parabolic equations: theory and implementation}, SIAM, 2008.

\bibitem{ruijgrok1982completely}
{\sc T.~Ruijgrok and T.~Wu}, {\em A completely solvable model of the nonlinear
  boltzmann equation}, Physica A: Statistical Mechanics and its Applications,
  113 (1982), pp.~401--416.

\bibitem{saint2009hydrodynamic}
{\sc L.~Saint-Raymond}, {\em Hydrodynamic limits of the Boltzmann equation},
  vol.~1971, Springer, 2009.

\bibitem{xu2010local}
{\sc Y.~Xu and C.-W. Shu}, {\em Local discontinuous galerkin methods for
  high-order time-dependent partial differential equations}, Communications in
  Computational Physics, 7 (2010), pp.~1--46.

\end{thebibliography}
